\documentclass[12pt]{article}
\usepackage{graphicx}
\usepackage{epsfig}
\usepackage{amsmath}
\usepackage{amssymb}
\usepackage{theorem}

\numberwithin{equation}{section}

\newtheorem{thm}{Theorem}[section]
\newtheorem{lemma}[thm]{Lemma}

\newtheorem{prop}[thm]{Proposition}
\newtheorem{cor}[thm]{Corollary}
{\theorembodyfont{\rmfamily}
\newtheorem{defn}[thm]{Definition}
\newtheorem{rmk}[thm]{Remark}
}

\newcommand{\qed}{\hfill \mbox{\raggedright \rule{.07in}{.1in}}}

\newenvironment{proof}{\vspace{1ex}\noindent{\bf
Proof}\hspace{0.5em}}{\hfill\qed\vspace{1ex}}
\newenvironment{pfof}[1]{\vspace{1ex}\noindent{\bf Proof of
#1}\hspace{0.5em}}{\hfill\qed\vspace{1ex}}

\newcommand{\R}{{\mathbb R}}
\newcommand{\C}{{\mathbb C}}
\newcommand{\Z}{{\mathbb Z}}
\newcommand{\D}{{\mathbb D}}
\newcommand{\N}{{\mathbb N}}

\renewcommand{\Re}{\operatorname{Re}}
\renewcommand{\Im}{\operatorname{Im}}

\newcommand{\SMALL}{\textstyle}

\title{First and higher order uniform dual ergodic theorems for dynamical systems \\ with infinite measure}

 \author{Ian Melbourne \thanks{Department of Mathematics, University of Surrey,
Guildford, Surrey GU2 7XH, UK}
 \and
 Dalia Terhesiu \thanks{
Department of Mathematics, University of Surrey,
Guildford, Surrey GU2 7XH, UK}
}

\date{25 May 2011}
\begin{document}

\maketitle

\begin{abstract}
We generalize the proof of  Karamata's Theorem by the method of approximation by polynomials  to the operator case. 
As a consequence, we offer a simple proof of \emph{uniform dual ergodicity} 
for a very large class of dynamical systems with infinite measure, and we obtain
bounds on the convergence rate.

In many cases of interest, including the Pomeau-Manneville family of 
intermittency maps, 
the estimates obtained through real Tauberian remainder theory are very
weak.    Building on the techniques of complex Tauberian remainder theory, we develop a method that provides \emph{second (and higher) order asymptotics}.
In the process, we 
derive a \emph{higher order Tauberian theorem} for scalar power series, which to our knowledge, has not previously been covered.
\end{abstract}

\section{Introduction and main results}
\label{sec-intro}

Suppose that  $(X,\mu)$ is an infinite measure space and $f:X\to X$ is
a conservative measure preserving transformation   with transfer operator 
$L:L^1(X)\to L^1(X)$. The transformation $f$ is {\em pointwise dual ergodic}  if there exists a positive sequence $a_n$ such that 
$a_n^{-1}\sum_{j=0}^{n-1}L^jv\to \int_X v\, d\mu$ a.e.\
as $n\to \infty$, for all $v\in L^1(X)$.
If furthermore, there exists $Y\subset X$ with $\mu(Y)\in(0,\infty)$ such that $a_n^{-1}\sum_{j=0}^{n-1}L^j 1_Y\to \mu(Y)$ uniformly on $Y$, then $Y$ is
referred to as a {\em Darling-Kac set}  (see Aaronson~\cite{Aaronson} for further background) and we refer to $f$ as {\em uniformly dual ergodic}.
At present, it is an open question whether every pointwise dual ergodic  transformation has a Darling-Kac set. However, it is desirable to prove  pointwise dual ergodicity by
identifying Darling-Kac  sets, as this facilitates the proof of several strong properties for $f$ (see for instance~\cite{Aaronson86,Aaronson, AaronsonDenker90,ADF92, ADU93, Thaler95,
ThalerZweimuller06,Zweimuller00} and for the setting of Markov chains~\cite{Bingham71, DarlingKac57, Lamperti62}; see also~\cite{BGT} and references therein).

\subsection{Uniform dual ergodicity}

An important class of examples
is provided by interval maps with indifferent fixed points (see for instance~\cite{Thaler83, Thaler95, Zweimuller00}). Standard examples are the family of Pomeau-Manneville intermittency
maps~\cite{PomeauManneville80} which are uniformly expanding except for 
an indifferent fixed point at $0$.  To fix notation, we focus on the version studied
by Liverani~{\em et al.}~\cite{LiveraniSaussolVaienti99}:
\begin{align} \label{eq-LSV}
fx=\begin{cases} x(1+2^\alpha x^\alpha), & 0<x<\frac12 \\ 2x-1, & \frac12<x<1
\end{cases}.
\end{align}
It is well known that for $\alpha\geq 1$, we are in the situation of infinite ergodic theory: there exist a unique (up to scaling) $\sigma$-finite, absolutely continuous invariant measure $\mu$, and $\mu([0,1])=\infty$. Let $\beta=1/\alpha$.  Using the standard procedure of inducing, several studies established that $Y=[\frac12,1]$ is a Darling-Kac set for $f$ with return sequence $a_n$ proportional to $n^\beta$ for $\beta\in (0,1)$, and proportional to $n/\log n$ for $\beta=1$.

An important refinement is the limit theorem of
Thaler~\cite{Thaler95} where the convergence of $a_n^{-1}\sum_{j=1}^nL^jv$
is shown to be uniform on compact subsets of $(0,1]$
for all observables of the form $v=u/h$ where $u$ is Riemann integrable
and $h$ is the density.   

The results of~\cite{Thaler95} are formulated for Markov maps of the interval
with indifferent fixed points. This includes transformations with slowly varying return sequences, the so called $\beta=0$ case. One  such example is
\begin{align} \label{eq-LSV0}
fx=\begin{cases} x(1+xe^{-1/x}), & 0<x<\frac12 \\ 2x-1, & \frac12<x<1
\end{cases},
\end{align}
with return sequence $a_n$ proportional to $\log n$ (see~\cite{Thaler83}).
Zweim\"uller~\cite{Zweimuller98,Zweimuller00} relaxed the Markov
condition and systematically studied non-Markovian nonuniformly
expanding interval maps (so-called AFN maps).  In 
particular,~\cite{Zweimuller98} obtained a spectral decomposition into
 basic (conservative and ergodic) sets and proved that for each basic set there
is a $\sigma$-finite absolutely continuous invariant measure, unique up to
scaling.  
The results in~\cite{Thaler95} on uniform dual
ergodicity were extended in~\cite{Zweimuller00} to the class of AFN maps.

In this paper, we generalize the proof of  Karamata's Theorem by the elementary method of approximation by polynomials~\cite{Karamata30,Karamata31} 
to the operator case. As a consequence, we offer a \emph{simple proof of uniform dual ergodicity} for a large class of dynamical systems with
infinite measure.
This method of proof, combined with techniques from~\cite{MT}, allows us to strengthen the results in~\cite{Thaler95, Zweimuller00}. 

It is convenient to describe our first result in the setting of AFN maps
$f:X\to X$,
though it applies to much more general systems, as described in Section~\ref{sec-upde-rem}.
Let $X'\subset X$ denote the complement of the indifferent fixed
points.   For any compact subset $A\subset X'$,
the construction in~\cite{Zweimuller98} yields 
a suitable first return set $Y$
containing $A$.   Fix such a set $Y$ with first return time function
$\varphi:Y\to\Z^+$, $\varphi(y)=\inf\{n\ge1:f^ny\in Y\}$.   We assume that the tail
probabilities are regularly varying:
$\mu(y\in Y:\varphi(y)>n)=\ell(n)n^{-\beta}$ where $\ell$ is slowly varying
and $\beta\in [0,1]$. 

For~\eqref{eq-LSV}, $\ell(n)$ is asymptotically constant
and $\beta=\frac{1}{\alpha}$.  For~\eqref{eq-LSV0}, $\ell(n)$ is asymptotically
proportional to $1/\log n$ and $\beta=0$.

Define $m(n)=\ell(n)$ for $\beta\in[0,1)$ and $m(n)=\tilde\ell(n)=\sum_{j=1}^n\ell(j)j^{-1}$
for $\beta=1$.     Set $D_\beta=\Gamma(1-\beta)\Gamma(1+\beta)$ for
$\beta\in(0,1)$ and $D_0=D_1=1$.   

\begin{thm} \label{thm-genTZ}
Suppose that $f:X\to X$ is an  AFN map with regularly
varying tail probabilities, $\beta\in[0,1]$.
Consider observables of the form $v=\xi u$ where $\xi$ is
$\mu$-integrable and bounded variation on $X$, and $u$ is Riemann integrable. Then 
\[
\lim_{n\to\infty}a_n^{-1}\sum_{j=0}^{n-1} L^j v=\int_X v\,d\mu
\]
uniformly on compact subsets of $X'$, where $a_n=D_\beta^{-1} n^\beta m(n)^{-1}$.
\end{thm}
The proof of Theorem~\ref{thm-genTZ} is provided in Subsection~\ref{sec-genupde}.

\begin{rmk}  Even in the case of AFN maps, the class of observables $v$ is much larger than in~\cite{Thaler95,Zweimuller00}.   
We note that Theorem~\ref{thm-genTZ} follows from our earlier paper~\cite{MT}
when $\beta\in(\frac12,1]$, 
but that the methods in~\cite{MT} fail for $\beta\in[0,\frac12]$.
\end{rmk}

\subsection{Convergence rates -- real Tauberian theory}

Karamata's approximation by polynomials method (generalized to the operator case) allows us to obtain convergence rates in Theorem~\ref{thm-genTZ}
by mimicking 
the arguments used in real Tauberian remainder theory (see for instance~\cite[Chapter~VII]{Korevaar} and references therein).
Below, we provide an example of such a theorem, restricting to the case of (\ref{eq-LSV0}), where the remainder is optimal. More general examples of Tauberian theorems with
remainders for positive operators are covered in Section~\ref{subsec-taub-rem}. In particular, Theorem~\ref{thm-upde-rem}(b) provides sharp remainders
for a large class of dynamical systems in the $\beta=0$ case.

\begin{thm} \label{thm-genT-rem}
Let $f$ be defined by (\ref{eq-LSV0}) and let $h$ be the density for $\mu$.
Set $c=\frac 12 h(\frac12)$.
Suppose that $v:[0,1]\to\R$ 
is H\"older or bounded variation supported on a compact subset of $(0,1]$. Then 
 \[
c\sum_{j=0}^{n-1} L^j v=\log n\int_0^1 v\,d\mu +O(1),
\] uniformly on compact sets of $(0,1]$.
\end{thm}
Theorem~\ref{thm-genT-rem} is proved in Section~\ref{sec-upde-rem}.

In the situation of (\ref{eq-LSV}), we obtain the error term $O(n^\beta/\log n)$  (see Theorem~\ref{thm-upde-rem}(a)).  The remainder in this case is not optimal (cf.~\cite[Corollary 9.3]{MT}). 
However, this result is probably the best possible using methods from real Tauberian theory, as discussed in Remark~\ref{rmk-real}.
Hence in the next subsection we turn to the complex theory.

\subsection{Higher order asymptotics -- complex Tauberian theory}

Building on the techniques of complex remainder theory for the scalar case (see Korevaar~\cite[Chapter~III.16]{Korevaar}), we develop a method that provides higher order uniform dual ergodic theorems for infinite measure preserving systems.  
For simplicity, in this paper we focus on the typical case of (\ref{eq-LSV}), but as explained in the sequel, our method applies to other cases of interest. Also, the case $\beta=1$  of (\ref{eq-LSV}) has been fully understood in~\cite{MT}; more general examples are considered in work in progress. Hence higher order theory for the case of $\beta=1$  is not considered in this paper.

\begin{thm} \label{thm-LSV-HO}
Let $f$ be defined by (\ref{eq-LSV}) with $\beta\in(0,1)$. 
Suppose that $v:[0,1]\to\R$ 
is H\"older or bounded variation supported on a compact subset of $(0,1]$.

Set $k=\max\{j\geq 0: (j+1)\beta-j>0\}$. 
Then for any $\epsilon>0$,
\[\sum_{j=0}^{n-1} L^j v=( C_0n^{\beta} +C_1 n^{2\beta-1}+C_2 n^{3\beta-2}+\cdots
+C_k n^{(k+1)\beta-k})\int_0^1\!\! v\,d\mu +O(n^{\epsilon}),\]uniformly on compact subsets of $(0,1]$,
where $C_1,C_2,\ldots$ are real nonzero constants (depending only on $f$). 
\end{thm}
Theorem~\ref{thm-LSV-HO} is proved in Section~\ref{sec-higher-order}.

\begin{rmk} \label{rmk-second}
In particular, when $\beta>\frac12$ we obtain the second order asymptotics $\lim_{n\to\infty}n^{1-\beta}(n^{-\beta}\sum_{j=0}^{n-1} L^j v- C_0 \int_0^1 v\,d\mu)=C_1 \int_0^1 v\,d\mu$.
\end{rmk}

\begin{rmk}
 In~\cite{MT}, we obtained first and higher order asymptotics of the iterates 
$L^n$ but the methods require $\beta>\frac12$.   As a byproduct, we obtained  
a weakened version of Theorem~\ref{thm-LSV-HO} with $n^\epsilon$ replaced by $n^\frac12$.   In particular, the second order asymptotics result in Remark~\ref{rmk-second} is contained in~\cite{MT} only for $\beta>\frac34$.
\end{rmk}

In the process of proving Theorem~\ref{thm-LSV-HO}, we obtain the following complex Tauberian theorem with remainder, which to our knowledge has not been previously considered. 

\begin{thm}\label{thm-taub} 
Let $\Phi(z)=\sum_{j=0}^\infty u_j z^j$ be a  convergent power series for $|z|<1$ with $|u_j|=O(1)$. Let  $1>\gamma_1>\gamma_2>\cdots>\gamma_k>0$, where $k\geq 0$.   Write $z=e^{-u+i\theta}$, $u>0$, $\theta\in [0, 2\pi)$. Suppose that 
\begin{equation*}
\Phi(z)=\sum_{r=1}^k A_r (u-i\theta)^{-\gamma_r} +O(1),\mbox{ as }u,\theta\to 0
\end{equation*} for $A_1,\dots,A_k$ real constants. Then for any $\epsilon>0$,
\begin{equation*}
\sum_{j=0}^{n-1}u_j=\sum_{r=1}^k\frac{A_r}{\Gamma(1+\gamma_r)}n^{\gamma_r}+O(n^\epsilon).
\end{equation*}
\end{thm}

\subsection{Renewal sequences}

Theorem~\ref{thm-taub} has immediate applications to  \emph{scalar renewal sequences with infinite mean}. To make this explicit we recall some basic background on scalar
renewal theory. For more details we refer the reader to~\cite{Feller66,BGT}. Let $(X_i)_{i\geq 1}$ be a sequence of  positive integer-valued independent identically distributed
random variables with probabilities $P(X_i=j)=f_j$.  
Define the partial sums  $S_n=\sum_{j=1}^n X_j$, and set $u_0=1$ and
$u_n=\sum_{j=1}^n f_j u_{n-j}$, $n\ge1$. Then it is easy to see that  $u_n=\sum_{j=0}^n P(S_j=n)$. The sequences $(u_n)_{n\geq 0}$ are called {\em renewal sequences}.

The analysis of
scalar renewal sequences with infinite mean relies crucially on the assumption of regularly varying tails: $\sum_{j>n}f_j = \ell(n)n^{-\beta}$, where $\ell$ is slowly varying and $\beta\in [0,1]$
(see~\cite{Feller66,BGT} and references therein). 
Then Karamata's Tauberian theorem yields
$\sum_{j=0}^n u_j\sim D_\beta^{-1}n^\beta m^{-1}(n)$, where $D_\beta$ and $m(n)$ are as defined above. 
(For results on first order asymptotics for $u_n$ we refer to~\cite{BGT, Erickson,GarsiaLamperti62}. )

A natural problem is to consider higher order expansions of $\sum_{j=0}^n u_j$.
Suppose for example that $\sum_{j>n}f_j = c n^{-\beta}+b(n)+c(n)$, where $b(n)=O(n^{-2\beta})$ and $c(n)$ is summable. If $\beta\leq 1/2$, assume further that $b(n)$ 
is monotone. Set $k=\max\{j\geq 0: (j+1)\beta-j>0\}$. 
It follows from the methods in this paper (specifically Theorem~\ref{thm-taub} together with a scalar version of Lemma~\ref{lem-zest-error}), that for any $\epsilon>0$,
\[
\sum_{j=0}^n u_j=( C_0n^{\beta} +C_1 n^{2\beta-1}+C_2 n^{3\beta-2}+\cdots
+C_k n^{(k+1)\beta-k}) +O(n^{\epsilon}), 
\]
where $C_1,C_2,\ldots$ are real nonzero constants.

Our method for proving uniform dual ergodic theorems centres around an operator
version of renewal sequences.
Let $(X,\mu)$ be a measure space (finite or infinite), and 
$f:X\to X$ a conservative 
measure preserving map.   Fix $Y\subset X$ with $\mu(Y)\in(0,\infty)$.
Let $\varphi:Y\to\Z^+$ be the first return time 
$\varphi(y)=\inf\{n\ge1:f^ny\in Y\}$ (finite almost everywhere by 
conservativity).  Let $L:L^1(X)\to L^1(X)$ denote the transfer operator 
 for $f$ and define
\[
T_n=1_YL^n1_Y,\enspace n\ge0, \qquad R_n=1_YL^n1_{\{\varphi=n\}},\enspace n\ge1.
\]
Thus $T_n$ corresponds to $u_n$ (returns to $Y$) and $R_n$ corresponds to $f_n$ 
(first returns).   The relationship $T_n=\sum_{j=1}^n T_{n-j}R_j$
generalises the notion of scalar renewal sequences.

Operator renewal sequences were introduced by Sarig~\cite{Sarig02} to study
lower bounds for mixing rates associated with finite measure preserving  systems, and
this technique was substantially extended and refined by Gou{\"e}zel~\cite{Gouezel04, Gouezel05}. The authors~\cite{MT} developed a theory of
renewal operator sequences for dynamical systems with infinite measure, generalizing the results of~\cite{GarsiaLamperti62,Erickson} to the operator case and obtaining mixing rates for a large class of systems including \eqref{eq-LSV} for $\beta\in(\frac12,1]$.
The uniform dual ergodic theorems proved in this paper follow from an operator version of Theorem~\ref{thm-taub}, namely  Theorem~\ref{beta>half}.

\vspace{1ex}
The rest of the paper is organised as follows.
In Section~\ref{sec-ORT}, we describe the general framework
for our results on the renewal operators $T_n$.
In Section~\ref{sec-upde-rem}, we generalize the proof of  Karamata's Theorem by the elementary method of approximation by polynomials  and obtain  uniform dual ergodic theorems with remainders  for a  large class of dynamical systems with
infinite measure.
Section~\ref{sec-higher-order} is devoted to higher order uniform dual ergodic theorems for a large class of systems.
Appendix~\ref{app-proofs} contains the proof of several technical
results stated in Section~\ref{sec-ORT}. 
Appendix~\ref{app-cont-int} contains computations of some complex contour integrals. 
In  Appendix~\ref{app-tails}, we provide explicit tail probabilities for (\ref{eq-LSV}) and (\ref{eq-LSV0}).

\vspace{-2ex}
\paragraph{Notation}
We use ``big O'' and $\ll$ notation interchangeably, writing
$a_n=O(b_n)$ or $a_n\ll b_n$ as $n\to\infty$ if there is a constant
$C>0$ such that $a_n\le Cb_n$ for all $n\ge1$.

\section{General framework}
\label{sec-ORT}

Let $(X,\mu)$ be an infinite measure space, and $f:X\to X$ a conservative
measure preserving map.   Fix $Y\subset X$ with $\mu(Y)=1$.
Let $\varphi:Y\to\Z^+$ be the first return time $\varphi(y)=\inf\{n\ge1:f^ny\in Y\}$
and define the first return map $F=f^\varphi:Y\to Y$.

The return time function $\varphi:Y\to\Z^+$ satisfies $\int_Y
\varphi\,d\mu=\infty$.  
We require that
\begin{itemize}
\item[]
$\mu(y\in Y:\varphi(y)>n)=\ell(n)n^{-\beta}$ where $\ell$ is slowly varying and
$\beta\in[0,1]$.
\end{itemize}

Recall that the transfer operator $R:L^1(Y)\to L^1(Y)$ for the first return
map $F:Y\to Y$ is defined via the formula
$\int_Y Rv\,w\,d\mu = \int_Y v\,w\circ F\,d\mu$, $w\in L^\infty(Y)$.
Let $\D=\{z\in\C:|z|<1\}$ and
$\bar\D=\{z\in\C:|z|\le1\}$.
Given $z\in\bar\D$, we define $R(z):L^1(Y)\to L^1(Y)$ to be
the operator $R(z)v=R(z^\varphi v)$.   Also, for each $n\ge1$, we define
$R_n:L^1(Y)\to L^1(Y)$, $R_nv=R(1_{\{\varphi=n\}}v)$.
It is easily verified that $R(z)=\sum_{n=1}^\infty R_nz^n$.

Our assumptions on the first return map $F:Y\to Y$ are
functional-analytic. We assume that there is a function
space $\mathcal{B}\subset L^\infty(Y)$ containing constant
functions, with norm $\|\;\|$ satisfying $|v|_\infty\le \|v\|$ for
$v\in\mathcal{B}$, such that 
\begin{itemize}
\item[(H1$'$)]
$R_n:\mathcal{B}\to\mathcal{B}$ are bounded linear operators satisfying
$\sum_{n=1}^\infty\|R_n\|<\infty$.
\end{itemize}
It follows that $z\mapsto R(z)$ is a continuous family of bounded linear operators
on $\mathcal{B}$ for $z\in\bar\D$. Since $R(1)=R$ and $\mathcal{B}$
contains constant functions, $1$ is an eigenvalue of $R(1)$.
Throughout, we assume:
\begin{itemize}
\item[(H2)] The eigenvalue $1$ is simple and isolated in the spectrum of
$R(1)$.
\end{itemize}
In Section~\ref{sec-upde-rem}, we prove that uniform dual ergodicity 
holds under (H1$'$) and (H2).
However, remainders and higher order theory require an improved version of (H1$'$).  
For simplicity, we will assume (H1) below, though many results can be obtained
under weaker assumptions.
\begin{itemize}
\item[(H1)] There is a constant $C>0$ such that
$\|R_n\|\le C\mu(\varphi=n)$ for all $n\ge1$.
\end{itemize}

\begin{rmk}
In~\cite{MT}, we studied the asymptotics of the iterates $L^n$
under conditions (H1) and (H2) together with an {\em aperiodicity} assumption
(H2(ii) in~\cite{MT}).   Moreover, generally the existence of an asymptotic expression for $L^n$ requires $\beta>\frac12$.
Uniform dual ergodicity deals with the asymptotics of $\sum_{j=1}^n L^j$, and
we prove in this paper that conditions (H1$'$) and (H2) suffice, with
no restriction on $\beta$.
\end{rmk}

Define the bounded linear operators
\[
T_n=1_YL^n1_Y,\enspace n\ge0, \qquad
R_n=1_YL^n1_{\{\varphi=n\}}=R1_{\{\varphi=n\}},\enspace n\ge1.
\]
(Here, $T_0=I$.) The power series
\[
T(z)=\sum_{n=0}^\infty T_n z^n,\enspace z\in\D,\qquad
R(z)=\sum_{n=1}^\infty R_n z^n,\enspace z\in\bar\D,
\]
are analytic on the open unit disk $\D$, and $R(z)$ is continuous on
$\bar\D$ by (H1).  We have the usual relation $T_n=\sum_{j=1}^n
T_{n-j}R_j$ for $n\ge1$, and it follows that $T(z)=I+T(z)R(z)$ on
$\D$.  Hence $T(z)=(I-R(z))^{-1}$ on $\D$.

\subsection{Asymptotics of $T(z)$ on $\D$}
\label{subsec-ass-T(z)}

Under an aperiodicity assumption, we obtained in~\cite{MT} the asymptotics of $T_n$ by estimating the Fourier coefficients of $T(e^{i\theta})$. 
The asymptotic expansion of $T(e^{i\theta})$ as $\theta\to 0$ is a key ingredient of the argument in~\cite{MT}.  

The corresponding key ingredients for the results in this paper
are the asymptotic expansion of $T(e^{-u})$ as $u\to0^+$ ($u$ real)
for first order uniform dual ergodicity with remainders, and
the asymptotic expansion of $T(z)$ as $z\to1$ ($z\in\D$) for higher order
results.

We recall the first order asymptotics of $T(e^{-u})$ from~\cite{MT}.
Denote the spectral projection corresponding to the simple eigenvalue $1$
for $R$ by $Pv=\int_Y v\,d\mu$.
\begin{prop}[\textbf{\cite[Proposition 7.1]{MT}}]
\label{prop-real-MT} Assume  (H1$'$) and (H2). Suppose
that $\mu(\varphi>n)=\ell(n) n^{-\beta}$ where $\ell$
is slowly varying and $\beta\in[0,1]$.  Define
$\tilde\ell(n) =\sum_{j=1}^n\ell(j)j^{-1}$.   Then
\[
T(e^{-u})\sim \begin{cases}
\tilde\ell(\frac{1}{u})^{-1}u^{-1}P, & \beta=1, \\
 \Gamma(1-\beta)^{-1}\ell(\frac{1}{u})^{-1}u^{-\beta}P, & \beta\in[0,1),
\end{cases}
\]
as $u\to 0^+$. 
(Recall that $A(x)\sim c(x) A$ for bounded linear operators $A(x),A:\mathcal{B}\to\mathcal{B}$ if $\|A(x)-c(x)A\|=o(c(x))$.)   \qed
\end{prop}

\begin{rmk}\cite[Proposition 7.1]{MT} does not contain the case $\beta=0$, but the proof in~\cite{MT} is easily extended to this case.
\end{rmk}

We now state three results that are proved in Appendix~\ref{app-proofs}.
The next  lemma provides the first  order expansion of $T(z)$ for $z\in\D$.
\begin{lemma} \label{lem-zestimate} 
Assume (H1$'$) and (H2).  Suppose
that $\mu(\varphi>n)=\ell(n) n^{-\beta}$ where $\ell$
is slowly varying and $\beta\in(0,1)$. Write $z= e^{-u+i\theta},  u>0$. 
Then 
\[
\Gamma(1-\beta)T(z)\sim \ell(1/|u-i\theta|)^{-1}(u-i\theta)^{-\beta}P,
\enspace\text{as $z\to1$}.
\]
\end{lemma}

Assuming (H1) and the existence of a good remainder in the tail probabilities $\mu(\varphi>n)$, one can also obtain higher order asymptotics (or asymptotics with remainder) of $ T(z)$ (see~\cite{MT} for higher order expansions of $T(e^{i\theta})$). 

\begin{lemma} \label{lem-zest-error} 
Assume (H1) and (H2) and let $\beta\in (0,1)$. Suppose that $\mu(\varphi>n)=c(n^{-\beta}+H(n))$, where $c>0$ 
and $H(n)=O(n^{-2\beta})$.   

\vspace{1ex}
If $\beta\in(\frac12,1)$, set
$c_H=-\Gamma(1-\beta)^{-1}\int_0^\infty H_1(x)\,dx$, where
$H_1(x)=[x]^{-\beta}-x^{-\beta}+H([x])$.
Let $k=\max\{j\ge0:(j+1)\beta-j>0\}$.
Then writing $z=e^{-u+i\theta}$, $u>0$,
\[
c\Gamma(1-\beta)T(z)=\Bigl( (u-i\theta)^{-\beta}+c_H(u-i\theta)^{1-2\beta}
+\dots+c_H^k(u-i\theta)^{k-(k+1)\beta}\Bigr)P+O(1).
 \]

\vspace{1ex}
If $\beta\in(0,\frac12]$, we assume further that $H(n)=b(n)+c(n)$
where $b(n)$ is monotone with $b(n)=O(n^{-2\beta})$ and $c(n)$ is summable.
Then
\[
c\Gamma(1-\beta)T(z)=(u-i\theta)^{-\beta}P+D(z),
\]
where $D(z)=O(1)$ for $\beta<\frac12$ and $D(z)=O(\log\frac{1}{|u-i\theta|})$ if $\beta=\frac12$.
\end{lemma}

Finally, we state two results about higher order expansions of $T(e^{-u})$
that go beyond the situations covered for $T(z)$ above.
We recall the following definition introduced
by de Haan~\cite{deHaan70} (see also~\cite[Chapter~3]{BGT}):

\begin{defn}(\textbf{\cite{deHaan70}}) A measurable function $f$ on $(0,\infty)$ is in the class $O\Pi_L$ for some slowly varying
function $L$ if for any $\alpha\geq 0$, $|f(\alpha x)-f(x)|=O(L(x))$ as $x\to\infty$.~\end{defn}
For example, if $f(x)=\log^px$, $p\in\R$, then $f\in O\Pi_L$ with $L=\log^{p-1}x$.

\begin{lemma}\label{lemma-HOT-beta0}
Assume (H1) and (H2). 
\begin{itemize}
\item[(a)]
Suppose that $\mu(\varphi>n)=cn^{-\beta}+O(n^{-\gamma})$ where $c>0$, 
$0<\beta<\gamma<1$.  Then 
\[
c\Gamma(1-\beta)T(e^{-u})=\bigl(u^{-\beta}+O(u^{\gamma-2\beta})\bigr)P+O(1),
\enspace u\in[0,1]
\]
\item[(b)] Let $\ell,\hat{\ell}$ be  slowly varying functions such that $\ell(x)\to\infty$ as $x\to\infty$, with $\ell$, $\hat\ell$, $\hat\ell^{-1}$ locally
bounded on $[0,\infty)$, and such that $\ell\in O\Pi_{\hat{\ell}}$.
Suppose that $\mu(\varphi>n)=\ell(n)^{-1}+O(\ell(n)^{-2}\hat\ell(n))$.  Then 
\[
T(e^{-u})=\bigl(\ell(1/u)+O(\hat\ell(1/u)\bigr)P+O(1),\enspace u\in[0,1].
\]
\end{itemize}
\end{lemma}

\begin{rmk}
The proof of Theorem~\ref{thm-genTZ} uses Proposition~\ref{prop-real-MT}, and the convergence
rates in Theorem~\ref{thm-genT-rem} (and related results) rely on
Lemma~\ref{lemma-HOT-beta0}.    Lemma~\ref{lem-zest-error} is required
for the asymptotic expansion in Theorem~\ref{thm-LSV-HO}.

Lemma~\ref{lem-zestimate} is included because it is clearly an interesting result in 
its own right even though we do not make explicit use of it in this paper.
We note that a proof of Theorem~\ref{thm-genTZ} for $\beta>0$ can be based on Lemma~\ref{lem-zestimate}
via the methods in Section~\ref{sec-higher-order}.   
However, the proof in Section~\ref{sub-uvKth} is more elegant.
Also, under an aperiodicity assumption, 
it is immediate from Lemma~\ref{lem-zestimate} that the Fourier coefficients
of $T(e^{i\theta})$ coincide with the $T_n$, bypassing the tedious calculation
in~\cite[Section~4]{MT}.
\end{rmk}

\section{Uniform dual ergodicity with remainders}
\label{sec-upde-rem}

In this section, we prove uniform dual ergodicity under hypotheses (H1$'$) and (H2). Assuming (H1) and (H2) and imposing further conditions on the tail probabilities, we obtain remainders in the implied convergence.

Recall that $m(n)=\ell(n)$ for $\beta\in[0,1)$, $m(n)=\tilde\ell(n)=\sum_{j=1}^n\ell(j)j^{-1}$
for $\beta=1$, $D_\beta=\Gamma(1-\beta)\Gamma(1+\beta)$ for
$\beta\in(0,1)$ and $D_0=D_1=1$.  The first result of this section reads as

\begin{thm}\label{thm-upde} Assume  (H1$'$) and (H2). Suppose $\mu(\phi>n)=\ell(n)n^{-\beta}$, where $\ell$ is slowly varying and $\beta\in[0,1]$. Then
\[
\lim_{n\to\infty} D_\beta m(n)n^{-\beta}\sum_{j=1}^n L^nv=\int_Yv\,d\mu,
\]
uniformly on $Y$ for all $v\in\mathcal{B}$.
\end{thm}

 The proof of Theorem~\ref{thm-upde} is provided in
Subsection~\ref{sub-uvKth}; it relies on a version of  Karamata's Tauberian Theorem that gives uniform
convergence for positive operators. This is the content of  Lemma~\ref{lem-unifKara}. 
In Subsection~\ref{sec-genupde}, we show that Theorem~\ref{thm-upde} extends
to a large class of observables supported on the whole of $X$ and we prove
Theorem~\ref{thm-genTZ}.

The proof of Lemma~\ref{lem-unifKara} combined with arguments used in~\cite[Theorems~3.1 and~3.2, Chapter~VII]{Korevaar} allows us to obtain uniform dual ergodic theorems with remainders.  Our next result provides an example. 

\begin{thm}\label{thm-upde-rem}
Assume  (H1) and (H2).   Let $v\in\mathcal{B}$.
\begin{itemize}
\item[(a)] If $\mu(\varphi>n)=c n^{-\beta}+O(n^{-\gamma})$, where $c>0$,  $\beta\in (0,1)$ and $\gamma>\beta$, then
\[
cD_\beta n^{-\beta}\sum_{j=1}^n 1_YL^jv=  \int_Yv\,d\mu  +E_nv,\quad
\|E_n\|=O(1/\log n).
\]

\item[(b)] 
Let $\ell,\hat{\ell}$ be  slowly varying functions such that $\ell(x)\to\infty$ as $x\to\infty$, with $\ell$, $\hat\ell$, $\hat\ell^{-1}$ locally
bounded on $[0,\infty)$, and such that $\ell\in O\Pi_{\hat{\ell}}$.
If $\mu(\varphi>n)=\ell(n)^{-1}+O(\ell(n)^{-2}\hat\ell(n))$, then 
\[
\sum_{j=1}^n 1_YL^jv=\{\ell(n)+O(\hat{\ell}(n))\}\int_Yv\,d\mu  +E_nv, \quad
\|E_n\|=O(1).
\]
\end{itemize}
\end{thm}
Theorem~\ref{thm-upde-rem} is proved in Subsection~\ref{subsec-taub-rem}.

\begin{pfof}{Theorem~\ref{thm-genT-rem}}
Choose $Y$ as in Proposition~\ref{prop-tail-PM0} so that
$\mu(\varphi>n)= c\log^{-1}n+O(\log^{-2} n)$.   
Take $\mathcal{B}$ to consist of H\"older or BV functions supported on $Y$ as appropriate.
In the notation of
Theorem~\ref{thm-upde-rem}(b), we have $\ell(n)=c^{-1}\log n$ and $\hat\ell=1$.  
It is well-known that hypotheses (H1) and (H2) are satisfied (see for example~\cite[Section~11]{MT}).    
Hence the result follows from Theorem~\ref{thm-upde-rem}(b).
 \end{pfof}

\subsection{A Karamata Theorem for positive operators}
\label{sub-uvKth}

Let $T(e^{-u})=\sum_{j=0}^\infty T_je^{-uj}$, $u>0$, where 
$T_j:\mathcal{B}\to\mathcal{B}$ are uniformly bounded positive operators.
Let $P:\mathcal{B}\to \mathcal{B}$ be a bounded linear operator.

\begin{prop} \label{prop-q}
Suppose that $T(e^{-u})\sim L(1/u)u^{-\beta}P$ as $u\to0^+$
where $\beta>0$ and $L$ is slowly varying.
Let $q:[0,1]\to\R$ be a polynomial satisfying $q(0)=0$.
Then
\begin{align*}
\Gamma(1+\beta)\sum_{j=0}^\infty T_j q(e^{-uj})\sim L(1/u)u^{-\beta}
\int_0^\infty q(e^{-t})\,dt^\beta \, P \quad\text{as $u\to0^+$},
\end{align*}
where $\sim$ is in the sense of the operator norm for
operators on $\mathcal{B}$.
\end{prop}

\begin{proof}
First note that $\int_0^\infty e^{-y}\,dy^\beta=\Gamma(1+\beta)$,
and so for any $k>0$,
\[
T(ku)\sim L(1/u)(ku)^{-\beta}P=(\Gamma(1+\beta))^{-1}L(1/u)u^{-\beta}
\int_0^\infty e^{-kt}\, dt^\beta \, P.
\]
Write $q(x)=\sum_{k=1}^m b_kx^k$.   Then
\begin{align*}
\sum_{j=0}^\infty T_j q(e^{-uj}) & =\sum_{k=1}^m b_k\sum_{j=0}^\infty T_j e^{-ukj}
=\sum_{k=1}^m b_k T(ku) \\
& \sim 
(\Gamma(1+\beta))^{-1}L(1/u)u^{-\beta} \sum_{k=1}^m b_k\int_0^\infty e^{-kt}\, dt^\beta \, P \\ & =
(\Gamma(1+\beta))^{-1}L(1/u)u^{-\beta} \int_0^\infty q(e^{-t})\, dt^\beta \, P.
\end{align*}

\vspace*{-2ex}
\end{proof}

\begin{prop} \label{prop-gq}
Define $g:[0,1]\to[0,1]$, $g=1_{[e^{-1},1]}$.    Let $\epsilon>0$.
Then there is a polynomial $q$ with $q(0)=0$ such that $q\ge g$ on $[0,1]$ and
\[
\int_0^1 (q(x)-g(x))x^{-3/2}\,dx<\epsilon.
\]
\end{prop}

\begin{proof}  We follow the argument starting at the bottom of page 21
in~\cite{Korevaar}.  Choose $\delta>0$ such that $\delta<(2e)^{-1}$ and
$\delta<((2e)^{3/2}+4)^{-1}\epsilon$.
Let $h:[0,1]\to[0,1]$ be the continuous function that (i) coincides
with $g$ except on the interval $J=[e^{-1}-\delta,e^{-1}]$,
and (ii) is linear on the interval $J$.
By the Weierstrass approximation theorem, there is a polynomial $R$ that 
$\delta$-approximates the continuous function $x^{-1}h(x)+\delta$ on
$[0,1]$.    Then $q(x)=xR(x)$ satisfies $q(0)=0$ and $q(x)\ge h\ge g$.
Also $\int_0^1 (h(x)-g(x))x^{-3/2}\,dx\le (2e)^{3/2}\delta$
and $\int_0^1 (q(x)-h(x))x^{-3/2}\,dx\le 2\delta\int_0^1 x^{-1/2}\,dx= 4\delta$.  The result follows.
\end{proof}

\begin{lemma}   \label{lem-unifKara}
Suppose that $T(e^{-u})\sim L(1/u)u^{-\beta}P$ as $u\to0^+$
where $\beta\ge0$ and $L$ is slowly varying.  Then 
\[
\lim_{n\to\infty}\Gamma(1+\beta)L(n)^{-1}n^{-\beta}\sum_{j=0}^n T_jv = Pv,
\]
uniformly on $Y$ for all $v\in\mathcal{B}$.
\end{lemma}

\begin{proof}
First, we consider the case $\beta>0$.

Without loss, we may suppose that $v\ge0$.
Define $g=1_{[e^{-1},1]}$, and given 
$\epsilon>0$ choose $q$ as in Proposition~\ref{prop-gq}.
By positivity of the operators $T_j$ and Proposition~\ref{prop-q}, 
\begin{align*}
\sum_{j=0}^n T_j & =
\sum_{j=0}^\infty  T_j g(e^{-j/n})\le
\sum_{j=0}^\infty  T_j q(e^{-j/n})
\sim (\Gamma(1+\beta))^{-1}L(n)n^\beta\int_0^\infty q(e^{-t})\,dt^\beta \,P,
\end{align*}
as $n\to\infty$.
More precisely, there exists $h_0(n)=o(1)$ as $n\to\infty$ such that for
any $v\in\mathcal{B}$, $v\ge0$, and any $y\in Y$,
\[
\sum_{j=0}^n (T_jv)(y)\le 
(\Gamma(1+\beta))^{-1}L(n)n^\beta\int_0^\infty q(e^{-t})\,dt^\beta \,\Bigl((Pv)(y)+h_0(n)\|v\|\Bigr).
\]
But
\begin{align*}
\int_0^\infty q(e^{-t})\,dt^\beta & =
\int_0^\infty g(e^{-t})\,dt^\beta+ 
\int_0^\infty (q(e^{-t})-g(e^{-t}))\,dt^\beta 
\\ & =1+ \int_0^1 (q(x)-g(x))\beta(\log x^{-1})^{\beta-1}x^{-1}\,dx
\\ & \le 1+ C\int_0^1 (q(x)-g(x))x^{-3/2}\,dx
< 1+C\epsilon,
\end{align*}
where $C$ is a constant depending only on  $\beta$.  Hence,
\[
(\Gamma(1+\beta)L(n)^{-1}n^{-\beta}\sum_{j=0}^n (T_jv)(y)\le (1+C\epsilon)((Pv)(y)+h_0(n)\|v\|).
\]
The reverse inequality is proved in the same way, and 
$\epsilon$ is arbitrary completing the proof for $\beta>0$.

Finally, we indicate the differences when $\beta=0$.
The conclusion of Proposition~\ref{prop-q} is replaced by 
\begin{align*}
\sum_{j=0}^\infty T_j q(e^{-uj})\sim L(1/u)
q(1)\, P \quad\text{as $u\to0^+$}.
\end{align*}
Let $q(x)=-7x^2+8x$.   Then $q(0)=0$, $q(1)=1$, and $q\ge g$ on $[0,1]$.
Hence
\begin{align*}
\sum_{j=0}^n T_j & =
\sum_{j=0}^\infty  T_j g(e^{-j/n})\le
\sum_{j=0}^\infty  T_j q(e^{-j/n})
\sim L(n) q(1) \,P=
L(n) \,P.
\end{align*}
For the reverse direction, take $q(x)=2x^2-x$ so that
$q(0)=0$, $q(1)=1$, and $q\le g$ on $[0,1]$.
\end{proof}

\begin{pfof}{Theorem~\ref{thm-upde}}
This follows immediately from
Proposition~\ref{prop-real-MT} and
Lemma~\ref{lem-unifKara}.
\end{pfof}

\subsection{Uniform dual ergodic theorems for general observables}
\label{sec-genupde}

Theorem~\ref{thm-upde} establishes uniform dual ergodicity for observables 
$v\in\mathcal{B}$.    In particular, it is required that $v$ is supported on $Y$.
However, it is straightforward to generalise~\cite[Theorem~10.2]{MT},
and thereby prove uniform dual ergodicity for a large class
of observables supported on the whole of $X$.

Let $X_k=f^{-k}Y\setminus \cup_{j=0}^{k-1}f^{-j}Y$.  Thus, $z\in X_k$ if and only if $k\geq 0$ is the smallest $k$ such that $f^k z\in Y$ (In particular, $X_0=Y$.).
Given $v\in L^\infty(X)$,  define $v_k=1_{X_k}v$.  Let
$\mathcal{B}(X)$ consist of functions $v\in L^1(X)$ satisfying 
$L^kv_k\in\mathcal{B}$ for each $k\geq 0$.

\begin{thm} \label{thm-genupde}
Assume the set up of Theorem~\ref{thm-upde}.  Let $v\in\mathcal{B}(X)$
and suppose that $\sum\|L^kv_k\|<\infty$.   Then 
\[
\lim_{n\to\infty} D_\beta m(n)n^{-\beta}\sum_{j=1}^n L^nv=\int_Xv\,d\mu, \quad
\text{uniformly on $Y$}.
\]
\end{thm}

\begin{proof}
Set $a_n=D_\beta^{-1} m(n)^{-1} n^\beta$.
By Theorem~\ref{thm-upde}, $\sum_{j=0}^n T_j=a_nP+D_n$, where $\|D_n\|=o(a_n)$.  
On $Y$,
\[
\sum_{j=0}^n L^j v=\sum_{j=0}^n \sum_{\ell=0}^j T_{j-\ell}L^\ell v_\ell=\sum_{\ell=0}^n \sum_{j=0}^{n-\ell}T_j L^\ell v_\ell.
\]
Thus,
\begin{align*}
\Bigl\|a_n^{-1}\sum_{j=0}^n L^j v-\int_X v\,d\mu\Bigr\| & \ll \Bigl\|a_n^{-1}\sum_{\ell=0}^n a_{n-\ell}\int_X L^\ell v_\ell\,d\mu-\sum_{\ell=0}^n\int_X v_\ell\,d\mu\Bigr\| \\
& \qquad\qquad +a_n^{-1}\sum_{\ell=0}^n\|D_{n-\ell}\|\|L^\ell v_\ell\|+ \sum_{\ell>n}\int_X |v_\ell|\,d\mu
\end{align*}
\begin{align*}
& =\sum_{\ell=0}^n \Big(\frac{a_{n-\ell}}{a_n}-1\Big)\int_X v_\ell\,d\mu+\sum_{\ell=0}^n a_n^{-1}\|D_{n-\ell}\|\|L^\ell v_\ell\|+ \sum_{\ell>n}\int_X |v_\ell|\,d\mu.
\end{align*}
Clearly, the last term converges to zero. Moreover, $\lim_{n\to\infty}(a_n^{-1}a_{n-\ell}-1)=0$, for each~$\ell$. Since we also know that  $\int_X v_\ell\,d\mu$ is summable, the first term
converges to zero. 
Finally, $\lim_{n\to\infty}a_n^{-1}\|D_{n-\ell}\|=0$, for each $\ell$. This together with the summability of  
$\|L^\ell v_\ell\|$ shows that the second term converges to
zero, ending the proof.
\end{proof}

\begin{pfof}{Theorem~\ref{thm-genTZ}} 
We modify the proof of~\cite[Theorem 11.5]{MT}. Without loss, choose $Y$ such that $f(Y)=X$. Suppose first that $v:X\to\R$ is
$\mu$-integrable and BV.
As noted in~\cite[Section 11.3]{MT}, the space $\mathcal{B}=BV(Y)$ is a suitable Banach space.
In particular, hypotheses (H1) and (H2) are satisfied and moreover
\mbox{$\sum \|L^k v_k\|<\infty$}.
Hence the hypotheses of Theorem~\ref{thm-genupde} are satisfied
and we obtain a uniform dual ergodic theorem for such $v$.
Finally, given $v$ of the required form $v=\xi u$, we approximate $u$ from above and below by BV functions $u^\pm$. Then $v$ is approximated from above and below by observables $v^\pm=\xi u^\pm$ for which uniform convergence on compact sets of $X'$ holds.  But $\int (v^+-v^-) $ can be made arbitrarily small and the result follows.~
\end{pfof}

\subsection{Remainders in the uniform convergence}
\label{subsec-taub-rem}

The proof of Theorem~\ref{thm-upde-rem} relies on the following Tauberian remainder theorem for positive operators.
Let $T(e^{-u})$, $P$ be bounded linear operators on $\mathcal{B}$,
$u\in(0,1)$.

\begin{lemma}\label{lemma-taub-rem}
\noindent(a) Let $\beta>0$, $\hat\beta<\beta$.
Suppose that $T(e^{-u})=u^{-\beta}P+D(u)$ as $u\to0^+$, where 
$\|D(u)\|=O(u^{-\hat\beta})$. Then
\[
\Gamma(1+\beta)\sum_{j=0}^n T_j= n^\beta P +E_n,
\quad \|E_n\|=O(n^\beta/\log n).
\]  

\noindent(b)  
Let $\ell,\hat{\ell}$ be  slowly varying functions such that $\ell(x)\to\infty$ as $x\to\infty$ and $\ell\in O\Pi_{\hat{\ell}}$.
Suppose that $T(e^{-u})= \{\ell(1/u)+O(\hat\ell(1/u))\}P+D(u)$ as $u\to0^+$,  
where $\|D(u)\|=O(1)$.  Then 
\[
\sum_{j=0}^n T_j= (\ell(n)+O(\hat\ell(1/n))P +E_n,
\quad \|E_n\|=O(1).
\] 
\end{lemma}

\begin{proof} We mimic the arguments used in the proof of~\cite[Theorem~3.1, Chapter~VII]{Korevaar}. Define $g:[0,1]\to[0,1]$, $g=1_{[e^{-1},1]}$.  By~\cite[Theorem~3.4, Chapter~VII]{Korevaar} (see also~\cite{Korevaar54a, Korevaar54b}),  there exist 
constants $m_0\ge1$, $C_1>0$, $C_2>1$, and for every $m\geq m_0$ there exist real polynomials 
$q^j_m(x)=\sum_{k=1}^m b^j_{km} x^k$, $j=1,2$, such that
\begin{itemize}
\item[(i)] $q^1_m(x)\leq g(x)\leq q^2_m(x)$, $0\leq x\leq 1$.
\item[(ii)] $\int_0^\infty (q^2_m(x)-q^1_m(x))\,dt^\beta\leq C_1/m$.
\item[(iii)] $\sum_{k=1}^m|b^j_{km}|\leq C_2^m$ for $j=1,2$.
\end{itemize}

Proceeding as in the proof of Lemma~\ref{lem-unifKara},  in case $(a)$, we have
\begin{align*}
\sum_{j=0}^nT_j &=\sum_{j=0}^\infty T_jg(e^{-j/n})
\leq \sum_{j=0}^\infty T_jq^2_m(e^{-j/n})
=\sum_{k=1}^m b^2_{km} T(k/n)
\\ & =n^\beta\sum_{k=1}^m b^2_{km}k^{-\beta}P+\sum_{k=1}^m b^2_{km}D(k/n).
\end{align*}
Moreover,
\begin{align*}
\Gamma(1+\beta)\sum_{k=1}^m b^2_{km}k^{-\beta} & = 
\sum_{k=1}^m b^2_{km}{\SMALL\int}_0^\infty e^{-kt}\,dt^\beta
 = {\SMALL\int}_0^\infty q^2_m(e^{-t})\,dt^\beta  \\
 & = 1+{\SMALL\int}_0^\infty( q^2_m(e^{-t})-g(e^{-t}))\,dt^\beta.  
\end{align*}
Hence
\begin{align*}
\Gamma(1+\beta)n^{-\beta}\sum_{j=0}^n T_j
& \le P+ {\SMALL\int}_0^\infty ( q^2_m(e^{-t})-g(e^{-t}))\,dt^\beta P 
+\Gamma(1+\beta)n^{-\beta}\sum_{k=1}^m b^2_{km}D(k/m).
\end{align*}
By property (i) and (ii) above, 
$\int_0^\infty ( q^2_m(e^{-t})-g(e^{-t}))\,dt^\beta\le C_1/m$.
By property (iii) and the hypothesis on $D(u)$, 
\[
\Bigl\|\sum_{k=1}^m b^2_{km}D(k/m)\Bigr\|\ll n^{\gamma}\sum_{k=1}^m k^{-\gamma}|b^2_{km}|
\le 
n^{\gamma}\sum_{k=1}^m |b^2_{km}|\le C_2^m n^{\gamma}.
\]
Hence, taking $m=\delta \log n$ with $\delta$ sufficiently small, we have
$\Gamma(1+\beta)n^{-\beta}\sum_{j=0}^n T_j-P\le H_n$ where
$\|H_n\|\ll 1/\log n$.
Repeating the argument with $q^1_m$ instead of $q^2_m$, we obtain the inequality in the reverse direction. 

In case $(b)$, we use two fixed polynomials $q^j(x)=\sum_{k=1}^m b_k^jx^k$, $j=1,2$
as in the proof of Lemma~3.5, so that $q^1\le g\le q^2$ in $[0,1]$
and $q^j(0)=0$, $q^j(1)=1$.
Write $T(e^{-u})= \{\ell(1/u)+h(u)\}P+D(u)$ as $u\to0^+$,  
where $h(u)=O(\hat\ell(1/u))$.
Then
\begin{align*}
\sum_{j=0}^nT_j & \leq  \sum_{k=1}^m b^2_k \ell(n/k)P
+\sum_{k=1}^m b^2_k h(k/n)P 
+\sum_{k=1}^m b^2_k D(k/n) \\
& = \ell(n) \sum_{k=1}^m b^2_k P + 
\sum_{k=1}^m b^2_k(\ell(n/k)-\ell(n))P
+\sum_{k=1}^m b^2_kh(k/n)P + H_n 
\\ &  = \ell(n)P + O(\hat\ell(n))P + H_n,
\end{align*}
where $\|H_n\|=O(1)$.
The reverse inequality is obtained by repeating the argument with $q^1$ instead of $q^2$.
\end{proof}

\begin{pfof}{Theorem~\ref{thm-upde-rem}}
This is immediate from 
Lemma~\ref{lemma-HOT-beta0} and Lemma~\ref{lemma-taub-rem}.
\end{pfof}

\begin{rmk}   \label{rmk-real}
As mentioned in the introduction, it seems unlikely that the apparently
weak estimate in Theorem~\ref{thm-upde-rem}(a) 
can be improved using methods from real Tauberian theory.    We explain this now
within the context of the classical scalar theory.

The method of approximation by polynomials links information about the asymptotics of real power series $\Phi(s)=\sum_{n=0}^\infty u_n e^{-ns}$ as $s\to 0^+$, to asymptotics of the partial sums $\sum_{j=0}^n u_j$ as $n\to\infty$ (see Korevaar~\cite[Chapters~I and~VII]{Korevaar}). 
Freud~\cite{Freud51} refined Karamata's method and obtained error estimates
of the type described in Theorem~\ref{thm-upde-rem}(a).   Within the
context of real power series, the Freud remainder terms are optimal.
For example, suppose that the coefficients $u_n$ are non-negative and
$\Phi(s)=s^{-\beta}+O(s^{-\gamma})$ as $s\to 0^+$, where $\beta>0$ and $\gamma<\beta$. By Ingham~\cite{Ingham65} and Korevaar~\cite{Korevaar54a, Korevaar54b}, the best possible result for the asymptotics of the partial sums
 is $\sum_{j=0}^n u_j=n^{\beta}+O(n^\beta/\log n)$  as $n\to\infty$.

More generally, Tauberian theorems relate the 
asymptotics of a nondecreasing function $U(x)$ as $x\to\infty$ with the
asymptotics of the
Laplace transform $\hat{U}(s)=\int_0^\infty e^{-sx}\, dU(x)$ as $s\to0^+$.
Improved remainder theorems can be obtained by 
imposing further conditions on $U(x)$
  (see for instance~\cite{Aljancic, Geluk, GelukdeHaan, deHaan,  Jordan, Omey85};  see also~\cite[Chapter~3]{BGT} and references therein).  
However, such conditions on $U$ are not appropriate in the context of dynamical systems.
\end{rmk}

\section{Higher order asymptotics }\label{sec-higher-order}

In this section, we prove a result on higher order asymptotics in the uniform dual
ergodic theorem.
Throughout,  we assume (H1) and (H2), and that $\mu(\varphi>n)=c(n^{-\beta}+H(n))$, where $\beta\in(0,1)$, $c>0$ 
and $H(n)=O(n^{-2\beta})$. 
If $\beta\in(0,\frac12]$, we assume further that $H(n)=b(n)+c(n)$
where $b(n)$ is monotone with $b(n)=O(n^{-2\beta})$ and $c(n)$ is summable.


 Recall  $H_1(x)=[x]^{-\beta}-x^{-\beta}+H([x])$.
For $\beta\in(\frac12,1)$,  
$c_H=-\Gamma(1-\beta)^{-1}\int_0^\infty H_1(x)\,dx$ and we set
$d_k= c_H^k/\Gamma\bigl((k+1)\beta-(k-1)\bigr)$ for $k\geq 0$.

\begin{thm}\label{beta>half}
Let $k=\max\{j\ge0:(j+1)\beta-j>0\}$. Then for any $\epsilon>0$,
\[c\Gamma(1-\beta)\sum_{j=0}^{n-1} T_j = (d_0n^{\beta} +d_1 n^{2\beta-1}+d_2 n^{3\beta-2}+\cdots
+d_k n^{(k+1)\beta-k})P +O(n^{\epsilon}).\] 
\end{thm}

\begin{pfof}{Theorem~\ref{thm-LSV-HO}}
Choose $Y$ as in Proposition~\ref{prop-tail-PM}.
Take $\mathcal{B}$ to consist of H\"older or BV functions supported on $Y$ as appropriate.
Again, it is well-known that hypotheses (H1) and (H2) are satisfied~\cite[Section~11]{MT}.
Hence the result follows from Theorem~\ref{beta>half} and Proposition~\ref{prop-tail-PM}.
\end{pfof}

In the remainder of this section, we prove Theorem~\ref{beta>half}.
The following result can be found in Korevaar~\cite[Proposition 16.1, Chapter III]{Korevaar},
see also~\cite{Subhankulov60}.

\begin{lemma}
Let $p\in\N$ fixed and let $0<r<1$. Then for all $\alpha\in(0,1]$ with $1-r\leq \alpha/4$ and for $m\in\Z$, the following holds:
 \begin{align}\nonumber
&\int_{-\alpha}^{\alpha}\frac{(e^{i\theta}-e^{i\alpha})^p(e^{i\theta}-e^{-i\alpha})^p }{1-re^{i\theta}}e^{im\theta}\, d\theta\\ &=
\begin{cases} O\Big(\frac{\alpha^{2p}}{\alpha^{p}|m|^p+1}\Big), & m\geq -(2p-1)\\
2\pi  r^{-(2p+m)}(1-2r\cos\alpha+r^2)^p +  O\Big(\frac{\alpha^{2p}}{\alpha^{p}|m|^p+1}\Big),  & m\leq -2p
\end{cases}
\end{align} 
where the implied constants are independent of $r,\alpha,m$.
\label{prop-Kor}
\end{lemma}

Recall $T(z)=\sum_{j=0}^{\infty}T_j z^j$ for $|z|<1$.  An immediate consequence
of the above proposition is

\begin{cor}[\textbf{cf~\cite[Corollary 16.2, Chaper~III]{Korevaar}}] 
Let $p\in\N$ be fixed.
Then for all $r\in(0,1)$, $\alpha\in (0,1]$ with $1-r\leq \alpha/4$, and $n\geq 2p$,
\begin{align}
 \nonumber 2\pi  r^{n-2p}(1-2r\cos\alpha+r^2)^p \sum_{j=0}^{n-2p} T_j =\int_{-\alpha}^{\alpha}\frac{T(re^{i\theta})}{1-re^{i\theta}}
(e^{i\theta}-e^{i\alpha})^p(e^{i\theta}-e^{-i\alpha})^pe^{-in\theta}
d\theta+ B,
\end{align}
where $||B||=O\Big(\sum_{j=0}^{\infty}
\frac{||T_j||r^j\alpha^{2p}}{\alpha^p|j-n|^p +1}\Big)$.
\label{part-cor}
\end{cor}

\begin{proof} This is exactly the same as the proof of \cite[Corollary
16.2, Chaper~III]{ Korevaar} formulated for the scalar case.
\end{proof}

Below we collect some useful tools for the proof of Theorem~\ref{beta>half}.

\begin{prop}\label{z-s} Let  $\theta\in [-\pi/2,\pi/2]$. Then $|1-e^{-1/n}e^{i\theta}|^{-1}\ll n$ and  $|1-e^{-1/n}e^{i\theta}|^{-1}\ll 1/|\theta|$.

\end{prop}

\begin{proof} To obtain the first estimate, we use the fact that $|1-e^{-1/n}e^{i\theta}|\geq\Re(1-e^{-1/n}e^{i\theta})\geq 1-e^{-1/n}$. Hence, $|1-e^{-1/n}e^{i\theta}|^{-1}\ll n$. 

To obtain the second estimate, we use the fact that $|1-e^{-1/n}e^{i\theta}|\geq |\Im(1-e^{-1/n}e^{i\theta})|= |\sin\theta|\gg|\theta|$. Hence, 
$|1-e^{-1/n}e^{i\theta}|^{-1}\ll 1/|\theta|$ .\end{proof}

\begin{prop}\label{A-s}Let $\gamma\in (0,1/2)$ fixed and let $A(n):=1-2e^{-1/n}\cos(1/n^\gamma)+e^{-2/n}$. Let
$A(\theta,n):=1-2e^{i\theta}\cos(1/n^\gamma)+e^{2i\theta}$ for $\theta\in [-1/n^\gamma, 1,n^\gamma]$. Then the following hold:

\begin{itemize}
\item[(a)] $|A^{-1}(n)|\ll n^{2\gamma}$ and  $|A^{-1}(n)A(\theta,n)|=O(1)$, uniformly in $\theta\in [-1/n^\gamma, 1/n^\gamma]$. 
%
%
\item[(b)] $|A^{-2}(n)(A^2(\theta,n)-A^2(n))|=O(1/n)$, uniformly in $\theta\in [-1/n, 1/n]$.
\item[(c)] If $\theta\in [-1/n^\gamma,-1/n]\cup [1/n, 1/n^\gamma]$ then $|A^{-2}(n)(A^2(\theta,n)-A^2(n))|\ll n^{2\gamma}\theta^{2\gamma+1}$.
\end{itemize}
\end{prop}

\begin{proof} (a) Write $A(n)=(1-e^{-1/n}e^{i/n^\gamma})(1-e^{-1/n}e^{-i/n^\gamma})$ and $A(\theta,n)=(e^{i\theta}-e^{i/n^\gamma})(e^{i\theta}-e^{-i/n^\gamma})$. Setting $\theta=\pm 1/n^\gamma$ in Proposition~\ref{z-s} yields the estimate for $A^{-1}(n)$. Clearly,  $|A(\theta,n)|\ll 1/n^{2\gamma}$ for  $\theta\in [-1/n^\gamma, 1,n^\gamma]$, yielding the second estimate.

%
$(b)$ Let $A_1(\theta,n):=A(\theta,n)-A(n)=(e^{i\theta}-e^{-1/n})(e^{i\theta}+e^{-1/n}-2\cos(1/n^\gamma))$. 
Since $\gamma<1/2$, $|A_1(\theta,n)|\ll (|\theta|+1/n)(|\theta|+1/n^{2\gamma})\ll n^{-2\gamma-1}$ uniformly in $\theta\in [-1/n,1/n]$.  Recall that $|A^{-1}(n)|\ll n^{2\gamma}$ and that $|A^{-1}(n)A(\theta,n)|=O(1)$.
Writing $A^2(\theta,n)-A^2(n)=(A(\theta,n)+A(n))A_1(\theta,n)$, we see that
\[|A^{-2}(n)(A^2(\theta,n)-A^2(n))|\ll |A^{-1}(n)A_1(\theta,n)|\ll n^{2\gamma} n^{-(2\gamma+1)}= 1/n.\]

$(c)$ If $\theta\in [-1/n^\gamma,-1/n]\cup [1/n, 1/n^\gamma]$, then $|A_1(\theta,n)|\ll |\theta|(|\theta|+1/n^{2\gamma})\ll|\theta|^{2\gamma+1}$.
Thus, $|A^{-2}(n)(A^2(\theta,n)-A^2(n))|\ll n^{2\gamma} |\theta|^{2\gamma+1}$.\end{proof}

\begin{pfof}{Theorem~\ref{beta>half}} We work out the implications of Corollary~\ref{part-cor}. Fix $p=2$, $\gamma\in(0,\beta)\cap(0,\frac12)$ and let
 $r=e^{-1/n}$,  
$\alpha=n^{-\gamma}$ (in particular, the constraint $1-r\leq\alpha/4$ is satisfied) and write
\begin{equation}\label{Sneq}
 \frac{2\pi}{e} e^{4/n}A^2(n) \sum_{j=0}^{n-4} T_j =\int_{-\alpha}^{\alpha}\frac{T(e^{-1/n} e^{i\theta})}{1-e^{-1/n}e^{i\theta}}A^2(\theta,n)
e^{-in\theta}\, d\theta +B,
\end{equation}
where \[||B||=O\Big(\sum_{j=0}^{\infty} \frac{||T_j||\alpha^4}{\alpha^2(j-n)^2 +1}\Big).\]

By Lemma~\ref{lem-zestimate}, $\|T(e^{-1/n} e^{i\theta})\|\ll |\frac1n-i\theta|^{-\beta}\le |\theta|^{-\beta}$. Hence, for all
$j\geq 0$,
\[
||T_j||\ll \int_{-\pi}^{\pi}|\theta|^{-\beta}\,d\theta=O(1).
\]
Thus,
\begin{equation}\label{sumop}
 ||B||\ll \alpha^4\Big(\sum_{|j-n|\leq 1/\alpha}\frac{1}{\alpha^2(j-n)^2 +1} +
\sum_{|j-n|\geq 1/\alpha}\frac{1}{\alpha^2(j-n)^2 +1}\Big)\ll \alpha^4 (1/\alpha)=n^{-3\gamma}.
\end{equation}
By Proposition~\ref{A-s}(a), $A^{-2}(n)\ll n^{4\gamma}$. Multiplying  (\ref{Sneq}) by $A^{-2}(n)$, we have
\begin{equation}\label{SnHO}
\frac{ 2\pi }{e}e^{4/n} \sum_{j=0}^{n-4} T_j =\int_{-n^{-\gamma}}^{n^{-\gamma}}\frac{T(e^{-1/n}e^{i\theta})}{(1-e^{-1/n}e^{i\theta})}\frac{A^2(\theta,n)}{A^{2}(n)}e^{-in\theta}\,d\theta+O(n^\gamma)=I+O(n^\gamma).
\end{equation}Next, write
\[I=\int_{-1/n^\gamma}^{1/n^\gamma}\frac{T(e^{-1/n} e^{i\theta})}{1-e^{-1/n}e^{i\theta}}
e^{-in\theta}\, d\theta +\int_{-1/n^\gamma}^{1/n^\gamma}\frac{T(e^{-1/n}e^{i\theta})}{1-e^{-1/n}e^{i\theta}}\frac{A^2(\theta,n)-A^2(n)}{A^2(n)}
e^{-in\theta}\, d\theta =J+H\]

We first estimate $H$.  Put $F(\theta,n):=T(e^{-1/n}e^{i\theta})(1-e^{-1/n}e^{i\theta})^{-1}A^{-2}(n)(A^2(\theta,n)-A^2(n))$.
Recall $\|T(e^{-1/n} e^{i\theta})\|\ll |\frac1n-i\theta|^{-\beta}\ll n^\beta$. By  Proposition~\ref{z-s},  $|1-e^{-1/n}e^{i\theta}|^{-1}\ll n$.
This together with  Proposition~\ref{A-s}(b)  implies
that $\|1_{[-1/n, 1/n]}F(\theta,n)\|\ll n^{\beta}$. Thus, $\|\int_{-1/n}^{1/n}F(\theta,n)\,d\theta\|\ll n^{\beta-1}$. 

By  Proposition~\ref{z-s},  $|1-e^{-1/n}e^{i\theta}|^{-1}\ll 1/|\theta|$. Hence, $\|T(e^{-1/n} e^{i\theta})(1-e^{-1/n}e^{i\theta})^{-1}\|\ll |\theta|^{-(\beta+1)}$. This together with 
 Proposition~\ref{A-s}(c)  implies that $\|1_a F(\theta,n)\|\ll n^{2\gamma}|\theta|^{2\gamma-\beta}$, where $a:=[-1/n^\gamma, -1/n]\cup [1/n, 1/n^\gamma]$. 
Thus, $\|\int_a F(\theta,n)\,d\theta\|\ll  n^{2\gamma}\int_a|\theta|^{2\gamma-\beta}\,d\theta\ll
n^{\gamma\beta+\gamma-2\gamma^2}$. Putting the two estimates together, we have $\|H\|\ll n^{\gamma(\beta+1-2\gamma)}$.

Next, we estimate $J$. Since $(1-e^{-1/n}e^{i\theta})^{-1}=(\frac1n-i\theta)^{-1}(1+O(|\frac1n-i\theta|))$,
\[J=\int_{-1/n^\gamma}^{1/n^\gamma}\frac{T( e^{-1/n}e^{i\theta})}{\frac1n-i\theta}
e^{-in\theta}\, d\theta +\int_{-1/n^\gamma}^{1/n^\gamma} T(e^{-1/n} e^{i\theta})h(\theta,n)\,d\theta,\] where $|h(\theta,n)|=O(1)$. 
But $\|\int_{-1/n^\gamma}^{1/n^\gamma} T(e^{-1/n} e^{i\theta})h(\theta,n)\,d\theta \|\ll \int_{-1/n^\gamma}^{1/n^\gamma}|\theta|^{-\beta}\,d\theta\ll n^{-\gamma(1-\beta)}$, so
$J=J_1+O(n^{-\gamma(1-\beta)})$ where 
$J_1=\int_{-1/n^\gamma}^{1/n^\gamma}\frac{T( e^{-1/n}e^{i\theta})}{\frac1n-i\theta}
e^{-in\theta}\, d\theta$.
By Lemma~\ref{lem-zest-error}, 
\[
c\Gamma(1-\beta)T(z)=\sum_{j=0}^k c_H^j ({\SMALL\frac1n}-i\theta)^{j-(j+1)\beta}P +D(z),
\] 
where in the worst case ($\beta=\frac12$), $\|D(z)\|=O(\log\frac{1}{|u-i\theta|})$. 
Thus, we can write 
\[
c\Gamma(1-\beta)J_1= \sum_{j=0}^k c_H^j L_j P
+J',\] where\[L_j=\int_{-1/n^\gamma}^{1/n^\gamma}\frac{e^{-in\theta}}{(\frac1n-i\theta)^{(j+1)\beta-(j-1)}}\,d\theta,
\]
and
\begin{align*}
\|J'\|&\ll\int_{-1/n^\gamma}^{1/n^\gamma}\frac{\log\frac{1}{|\frac1n-i\theta|}}{|\frac1n-i\theta|}\,d\theta
\ll\log n\int_0^{1/n}n\,d\theta+\log n\int_{1/n}^{1/n^\gamma}\frac{1}{\theta}\,d\theta\ll(\log n)^2.
\end{align*}
 By Corollary~\ref{contb} (with $\rho=(j+1)\beta-j$),
\[L_j=\frac{2\pi}{e}\frac{n^{(j+1)\beta-j}}{\Gamma((j+1)\beta-(j-1))}+O(n^{\gamma((j+1)\beta)-j)}).\]
Putting all the estimates together,
\[
c\Gamma(1-\beta)I=\frac{2\pi}{e}(\sum_{j=0}^k d_j n^{(j+1)\beta-j})P+ O(n^{\gamma(\beta+1-2\gamma)}).
\]
The conclusion follows by plugging this estimate into (\ref{SnHO}) and taking
$\gamma$ sufficiently small.
\end{pfof}


\begin{rmk}\label{rmk-compl-taub}
The proof of Theorem~\ref{thm-taub}  goes exactly same as the proof of Theorem~\ref{beta>half} with $T(z)$ replaced by $\Phi(z)$, $\sum_{j=0}^{n-1} T_j $  replaced by $\sum_{j=0}^{n-1} u_j$ and
$j-(j+1)\beta$ replaced by $\gamma_j$, $j=1,\ldots,k$. The asymptotics of $\Phi(z)$ is part of the hypothesis of Theorem~\ref{thm-taub}. 
\end{rmk}

\appendix
\section{Proof of Lemmas~\ref{lem-zestimate},~\ref{lem-zest-error} and~\ref{lemma-HOT-beta0}}
\label{app-proofs}

By (H1$'$) and (H2), there exists $\epsilon>0$ and a continuous family of simple
eigenvalues of $R(z)$, namely $\lambda(z)$ for $z\in \bar\D\cap B_\epsilon(1)$ with
$\lambda(1)=1$.  Let $P(z):\mathcal{B}\to\mathcal{B}$ denote the
corresponding family of spectral projections  with  $P(1)=P$ and complementary
projections $Q(z)=I-P(z)$. 
Also, let
$v(z)\in\mathcal{B}$ denote the corresponding family of
eigenfunctions normalized so that $\int_Y v(z)\,d\mu=1$ for all $z$.
In particular, $v(1)\equiv1$.

Then we can write
\begin{equation}\label{exprT}
T(z)=(1-\lambda(z))^{-1}P(z)+(I-R(z))^{-1}Q(z), 
\end{equation}
for $z\in \bar\D\cap B_\epsilon(1)$, $z\neq1$.

\begin{prop} \label{prop-PQ} 
Assume (H1$'$) and (H2). There exists $\epsilon,C>0$ such that
$\|(I-R(z))^{-1}Q(z)\|\le C$ for $z\in \bar\D\cap
B_\epsilon(1)$, $z\neq1$. \qed
\end{prop}

\begin{pfof} {Lemma~\ref{lem-zestimate}}
The main part of the proof is to show that
$1-\lambda(z)\sim \Gamma(1-\beta)\ell(1/|u-i\theta|)(u-i\theta)^\beta$.
By~\eqref{exprT}, continuity of $P(z)$, and Proposition~\ref{prop-PQ},
$T(z)=(1-\lambda(z))^{-1}(P+o(1))+O(1)$ and the result follows.

For the asymptotics of $1-\lambda(z)$, we follow~\cite{AaronsonDenker01,MT}.
Note that $|v(z)-1|_\infty=o(1)$ as $z\to 1$.
Define the distribution function $G(x)=\mu(\varphi\le x)$.
Choose $\epsilon>0$  such that $\lambda(z)$  is well defined for $z=e^{-u+i\theta}$ for $u> 0$ and $\theta\in (-\epsilon, \epsilon)$.
Then $\lambda(z)v(z)=R(z)v(z)=R(e^{(-u+i\theta)\varphi}v(z))$,
and so $\lambda(z)=\int_Y \lambda(z)v(z)\,d\mu=
\int_Ye^{(-u+i\theta)\varphi}v(z)\,d\mu
=1-\int_Y (1-e^{(-u+i\theta)\varphi})v(z)\,d\mu$.

Let $\mathcal{F}_0$ denote the $\sigma$-algebra generated by $\varphi$, 
namely the partition into the sets $\{\varphi=n\}$.   Define the step function 
$\hat v(z):[0,\infty)\to\C$ given by
$\hat v(z)\circ \varphi=E(v(z)|\mathcal{F}_0)$.   Then
$\lambda(z)=1-\int_0^\infty (1-e^{(-u+i\theta)x)})\hat{v}_z(x)\, dG(x)$, where $|\hat{v}_z(x)-1|=o(1)$, uniformly in $x$, as $z\to 1$.

Define $\hat{G}_z$ such that  $d\hat{G}_z=\hat{v}_z dG$. Integrating by parts,
\[
1-\lambda(z)=\int_0^\infty (1-e^{(-u+i\theta)x)})\, d\hat{G}_z(x)=
(u-i\theta)\int_0^\infty e^{-(u-i\theta)x}g_z(x)(1-G(x))\,dx,
\]
where $|g_z(x)-1|=o(1)$,  uniformly in $x$, as $z\to 1$.

Since $1-G(x)=x^{-\beta}\ell(x)$, we can write
$1-\lambda(z) =\ell(1/|u-i\theta|)(u-i\theta)^{\beta}J(z)$, where
\begin{align*}
J(z)=\int_0^\infty \frac{e^{-(u-i\theta)x}}{[(u-i\theta)x]^{\beta}}
h_{u,\theta}(x) g_{u,\theta}(x) (u-i\theta)\,dx, \quad
h_{u,\theta}(x)=\frac{\ell(x)}{\ell(1/|u-i\theta|)}.
\end{align*}
Here, 
$|g_{u,\theta}-1|_{\infty}=o(1)$ as $u,\theta\to 0$.  
The proof will be complete once we show that $J(z)$ converges to $\Gamma(1-\beta)$ as $u,\theta\to 0$.  
We consider only the case $\theta>0$, since the case $\theta<0$  follows by the same argument.

We consider separately each of two possible cases: (i) $0\leq\theta\leq u$; (ii) $0\leq u\le \theta$. 
In both cases we let $\hat{h}_{u,\theta}(\sigma)=h_{u,\theta}(\sigma/|u-i\theta|)$.   

In case $(i)$ we let 
$I_1  =\int_0^\infty e^{-(u-i\theta)x}[(u-i\theta)x]^{-\beta}(u-i\theta)\,dx$. 
Put $y=\theta/u$. Using the substitution $\sigma=|u-i\theta|x$,  we have
\begin{align*}
J(z)&=I_1+\Big( \frac{1- iy}{|1-iy|}\Big)^{1-\beta}\int _0^\infty e^{-\sigma|1-iy|^{-1}} e^{ i\sigma y|1-iy|^{-1}}\sigma^{-\beta}(\hat{h}_{u,\theta}(\sigma)-1)\, d\sigma\\
&\, \qquad +\Big( \frac{1- iy}{|1-iy|}\Big)^{1-\beta}\int _0^\infty e^{-\sigma|1-iy|^{-1}} e^{ i\sigma y|1-iy|^{-1}}\sigma^{-\beta}\hat{h}_{u,\theta}(\sigma)(g_{u,\theta}(\sigma/|u-i\theta|)-1)\, d\sigma\\
&= I_1 +\Big( \frac{1- iy}{|1-iy|}\Big)^{1-\beta}(I_2+I_3).
\end{align*}
By Proposition~\ref{first_cont},  we have $I_1=\Gamma(1-\beta)$. It remains to estimate $I_2$ and $I_3$.

 We show that $\lim_{u,\theta\to 0}I_2=0$ by applying the dominated convergence theorem (DCT). By Potter's bounds (see e.g.~\cite{BGT}), for any fixed $\delta>0$, 
$|\hat{h}_{u,\theta}(\sigma)|\ll\sigma^{\delta}+ \sigma^{-\delta}$,  for all $\sigma\in (0,\infty)$. Since $|1-iy|\in [1,\sqrt 2]$, the integrand of $I_2$ is bounded, up to a constant,  by 
the function $e^{-\sigma/\sqrt 2}(\sigma^{-(\beta-\delta)}+\sigma^{-(\beta+\delta)})$, which is $L^1$ when 
taking $\delta\in (0,1-\beta)$. 
Moreover, for each fixed $\sigma\in (0,\infty)$, $\lim_{u,\theta\to 0}\hat{h}_{u,\theta}(\sigma)=1$ by definition of $\ell$ being slowly varying. Thus, for each $\sigma$,
\[\lim_{u,\theta\to 0}\frac{e^{-\sigma|1-iy|^{-1}} e^{ i\sigma y|1-iy|^{-1}}}{\sigma^{\beta}}(\hat{h}_{u,\theta}(\sigma)-1)= 0.\] 
 It follows from the DCT
that $\lim_{u,\theta\to 0}I_2=0$.

Next, since for any $\delta>0$, $|\hat{h}_{u,\theta}(\sigma)|\ll\sigma^{\delta}+ \sigma^{-\delta}$,  for all $\sigma\in (0,\infty)$ and $|1-iy|\in [1,\sqrt 2]$,
we have that 
 $|I_3|\ll |g_{u,\theta}-1|_\infty \int_0^\infty e^{-\sigma/\sqrt 2}(\sigma^{-(\beta-\delta)}+\sigma^{-(\beta+\delta)})\,d\sigma\ll {|g_{u,\theta}-1|_\infty\to0}$, as $u,\theta\to 0$. 
 This ends the proof in case (i).

In case (ii), we write 
$\hat{v}_z=1+\hat{v}_z^1-\hat{v}_z^2+i\hat{v}_z^3-i\hat{v}_z^4$, where $\hat{v}_z^j\geq 0$ and $\sup_x|\hat{v}_z^j(x)|=o(1)$ as $z\to 1$, for $j=1,2,3,4$.

Define $\hat{G}_z^j$ such that  
$d\hat{G}_z^j=\hat{v}_z^j dG$. Integrating by parts,
\begin{align*}
1-\lambda(z) &=
(u-i\theta)\int_0^\infty (1+\sum_{j=1}^4 q_j\hat{g}_{u,\theta}^j(x))e^{-(u-i\theta)x}(1-G(x))\,dx,
\end{align*}
so we obtain
\begin{align*}
J(z)&=
\int _0^\infty E_{u,\theta}(x)h_{u,\theta}(x) \,dx
+\sum_{j=1}^4 q_j\int_0^\infty E_{u,\theta}(x)  \hat{g}_{u,\theta}^j(x)h_{u,\theta}(x)\,dx=
I+\sum_{j=1}^4 q_j I^j,
\end{align*}
where 
$E_{u,\theta}(x):=e^{-(u-i\theta)x}[(u-i\theta)x]^{-\beta}(u-i\theta)$,
$q_1=1, q_2=-1, q_3=i, q_4=-i$ and $\sup_x|\hat{g}_{u,\theta}^j(x)|=o(1)$ as $u,\theta\to 0$ for $j=1,2,3,4$.
Write
\begin{align*}
I&=\int_0^{b/\theta} E_{u,\theta}(x)\,dx +\int_0^{b/\theta} E_{u,\theta}(x)( h_{u,\theta}(x)-1)\,dx
+\int_{b/\theta}^\infty E_{u,\theta}(x) h_{u,\theta}(x)\,dx,
\end{align*} 
for some positive large $b$. Similarly, for $j=1,2,3,4$, write
 \[
 I^j=\int_0^{b/\theta} E_{u,\theta}(x)h_{u,\theta}(x) \hat{g}_{u,\theta}^j(x)\,dx +\int_{b/\theta}^\infty E_{u,\theta}(x) h_{u,\theta}(x) \hat{g}_{u,\theta}^j(x)\,dx.
\]

 Next,  put $y=u/\theta$. Using the substitution $\sigma=|u-i\theta|x$,  we  have
\begin{align*}
& \int_0^{b/\theta}E_{u,\theta}(x)( h_{u,\theta}(x)-1)\,dx
\\ & \qquad\qquad =\Big( \frac{y-i}{|y-i|}\Big)^{1-\beta}
\int _0^{|y-i|b} (\hat{h}_{u,\theta}(\sigma)-1)e^{-\sigma y|y-i|^{-1}} e^{ i\sigma |y-i|^{-1}}\sigma^{-\beta}\,d\sigma,
\end{align*}
and
\begin{align*}
& \int_0^{b/\theta}E_{u,\theta}(x)h_{u,\theta}(x) \hat{g}_{u,\theta}^j(x)\,dx \\ & \qquad\qquad =\Big( \frac{y-i}{|y-i|}\Big)^{1-\beta}
\int _0^{|y-i|b}e^{-\sigma y|y-i|^{-1}} e^{ i\sigma |y-i|^{-1}}\sigma^{-\beta}\hat{h}_{u,\theta}(\sigma) \hat{g}_{u,\theta}^j\Big(\frac{\sigma}{|u-i\theta|}\Big)\,d\sigma.\end{align*}

By the argument used for $I_2$ in case (i), we have  that $\lim_{u,\theta\to 0}\int_0^{b/\theta} E_{u,\theta}(x)( h_{u,\theta}(x)-1)\,dx=0$. Since $\hat{g}_{u,\theta}^j(\sigma/|u-i\theta|)\to 0$ uniformly in $\sigma$, the argument used for $I_3$ in case~(i)  shows that
  $\lim_{u,\theta\to 0}\int_0^{b/\theta} E_{u,\theta}(x)h_{u,\theta}(x) \hat{g}_{u,\theta}^j(x)\,dx =0$. 

The next step is to estimate 
\[
I_b^\infty:=\int_{b/\theta}^\infty E_{u,\theta}(x) h_{u,\theta}(x)\,dx+\sum_{j=1}^4 q_j\int_{b/\theta}^\infty E_{u,\theta}(x) h_{u,\theta}(x) \hat{g}_{u,\theta}^j(x)\,dx.
\] 
The substitution $\sigma=\theta x$ gives
 \begin{align*}I_b^\infty=(y-i)^{1-\beta}
\Big(\int_{b}^\infty\frac{e^{-\sigma y} e^{ i\sigma }}{\sigma^{\beta}}h_{u,\theta}(\sigma/\theta)\,d\sigma+
\sum_{j=1}^4 q_j\int_{b}^\infty\frac{e^{-\sigma y} e^{ i\sigma }}{\sigma^{\beta}}h_{u,\theta}(\sigma/\theta)\hat{g}_{u,\theta}^j(\sigma/\theta)\,d\sigma\Big).\end{align*}
But all the integrals are oscillatory  with $\sigma\mapsto     e^{-\sigma y} \sigma^{-\beta}h_{u,\theta}(\sigma/\theta)$ and 
$\sigma\mapsto     e^{-\sigma y}\sigma^{-\beta}h_{u,\theta}(\sigma/\theta)\hat{g}_{u,\theta}^j(\sigma/\theta)$, respectively,
decreasing for each fixed value of $u$ and $\theta$.    By Potter's bounds, $h_{u,\theta}(\sigma/\theta)\ll (\sigma|y-i|)^\delta +(\sigma|y-i|)^{-\delta}
\ll \sigma^\delta$ and so $I_b^\infty= O(b^{-(\beta-\delta)})$.

Putting all these together, $\lim_{u,\theta\to 0}J(z)=\lim_{u,\theta\to 0} \int_0^{b/\theta} E_{u,\theta}(x)\,dx  +O(b^{-(\beta-\delta)})$.
By Proposition~\ref{first_cont},  $\int_0^{b/\theta} E_{u,\theta}(x)\,dx =\Gamma(1-\beta)+O(b^{-\beta})$. This ends the proof in case~(ii), since $b$ is arbitrary. 
\end{pfof}

Next, we turn to the higher order expansions.
The following consequence of (H1) is
standard (see for instance~\cite[Proposition~2.7]{MT}).

\begin{prop} \label{prop-H1}Assume (H1) and that $\mu(\varphi>n)=\ell(n)n^{-\beta}$ where $\ell$ is slowly varying and  $\beta\in [0,1)$. Then
there is a constant $C>0$ such that $\|R(r e^{i(\theta+h)})-R(r
e^{i\theta})\|\le C \ell(1/h)h^\beta$ and $\|R(r)-R(1)\|\le C
\ell(\frac{1}{1-r})(1-r)^\beta$ for all $\theta\in[0,2\pi)$,
$r\in(0,1]$, $h>0$.  \qed
\end{prop}

\begin{cor}  \label{cor-H1}
The estimates for $R(z)$ in Proposition~\ref{prop-H1} are inherited by the
families $P(z)$, $Q(z)$, $\lambda(z)$ and  $v(z)$, where defined.
\qed
\end{cor}

\begin{lemma} \label{lem-lambda}
Assume (H1) and (H2) and let $\beta\in (0,1)$. 
Suppose that $\mu(\varphi>n)=c(n^{-\beta}+H(n))$, where $c>0$ 
and $H(n)=O(n^{-q})$, $q>\beta$.
If $q\le1$, we assume further that 
$H(n)=b(n)+c(n)$, where $b$ is monotone
with $|b(n)|=O(n^{-q})$ and $c(n)$ is summable.    

If $q>1$, set $c_H=-\Gamma(1-\beta)^{-1}\int_0^\infty H_1(x)\,dx$
where $H_1(x)=[x]^{-\beta}-x^{-\beta}+H([x])$.
If $q\le1$, set $c_H=0$.

Then writing  $z= e^{-u+i\theta}$, $u>0$, 
 \[
1-\lambda(z)=c\Gamma(1-\beta)\{(u-i\theta)^\beta-c_H(u-i\theta)
+O(|u-i\theta|^{2\beta})+D(z),
\]
where $D(z)= O(|u-i\theta|^q)$ if $q\neq1$, and 
$D(z)= O(|u-i\theta|\log\frac{1}{|u-i\theta|})$ if $q=1$.
\end{lemma}

\begin{proof}
We may suppose without loss that $q<\beta+1$.
The proof adapts the proof of Lemma~\ref{lem-zestimate}, noting that $\ell=c+o(1)$. 
The notation for the functions
$G$ and $g_{u,\theta}$ used in the proof of Lemma~\ref{lem-zestimate} keeps its meaning in the proof below. However, since (H1) holds
and  $1-G(x)=O(x^{-\beta})$, we have the more precise estimate
$\sup_{x\geq 0}|g_{u,\theta}(x)-1|\ll |u-i\theta|^\beta$.
Note also that $1-G(x)=c(x^{-\beta}+H_1(x))$ where
$H_1(x)=H(x)+O(x^{-(\beta+1)})=O(x^{-q})$.
As in the proof of Lemma~\ref{lem-zestimate}, we consider separately each of two possible cases: (i) $0\leq\theta\leq u$, (ii) $0\leq u\le \theta$.

In case (i), we have
\begin{align*}
1-\lambda(z) &=c(u-i\theta)^{\beta}\int _0^\infty \frac{e^{-(u-i\theta)x}}{[(u-i\theta)x]^{\beta}}
 (u-i\theta)\,dx+c(u-i\theta)\int_0^\infty e^{-(u-i\theta)x}H_1(x)\,dx
\\&+(u-i\theta)\int _0^\infty e^{-(u-i\theta)x}(1-G(x))
(g_{u,\theta}(x)-1) \,dx\\&=c(u-i\theta)^{\beta}I_1+c(u-i\theta)I_2+(u-i\theta)I_3.
\end{align*}

We already know that $I_1=\Gamma(1-\beta)$. Next, we estimate $I_3$
proceeding as for $I_3$ of case (i) in the proof of Lemma~\ref{lem-zestimate}.
Since we are in case (i), $|g_{u,\theta}|\ll u^\beta$, so
\[
|(u-i\theta)I_3|\ll u|g_{u,\theta}-1|_\infty \int_0^\infty e^{-ux}x^{-\beta}\,d\sigma\ll u^{2\beta}\int_0^\infty e^{-\sigma}\sigma^{-\beta}\, d\sigma \ll |u-i\theta|^{2\beta}.
\]
If $q<1$, then similarly
\begin{align*}
 |(u-i\theta)I_2| & \ll u\int_0^\infty e^{-ux}x^{-q}\,dx 
= u^q\int_0^\infty e^{-\sigma}\sigma^{-q}\,d\sigma \ll |u-i\theta|^q.
\end{align*}
If $q=1$, then
\begin{align*}
|(u-i\theta)I_2| & \ll  u + u\int_1^\infty e^{-ux}x^{-1}\,dx 
= u+ u\int_u^\infty e^{-\sigma}\sigma^{-1}\,d\sigma \ll u\log\frac1u
\\ & \le |u-i\theta|\log\frac{1}{|u-i\theta|}.
\end{align*}
Finally, for $q>1$,
\begin{align*}
 I_2=\int_0^\infty (e^{-(u-i\theta)x}-1)H_1(x)\,dx+\int_0^\infty H_1(x)\,dx=I_2'-c_H\Gamma(1-\beta),
\end{align*}
where
\begin{align*}
 |I_2'|\ll \Big(\int_0^{1/|u-i\theta|}|u-i\theta|x^{1-q}\,dx+\int_{1/|u-i\theta|}^\infty 1/x^q\,dx\Big)\ll|u-i\theta|^{q-1}.
\end{align*}

 In case (ii), we write
\begin{align*}
1-\lambda(z) &=c(u-i\theta)^{\beta}
\int _0^\infty \frac{e^{-(u-i\theta)x}}{[(u-i\theta)x]^{\beta}} (u-i\theta)\,dx+  c(u-i\theta)\int _0^\infty e^{-(u-i\theta)x}H_1(x)\,dx\\ 
&+(u-i\theta)\sum_{j=1}^4 q_j\int_0^\infty   e^{-(u-i\theta)x}(1-G(x)) \hat{g}_{u,\theta}^j(x)\,dx
\\& =(u-i\theta)^{\beta}I_1+(u-i\theta)I_2+(u-i\theta)\sum_{j=1}^4 q_j I^j,
\end{align*}
where $q_1=1, q_2=-1, q_3=i, q_4=-i$ and $\sup_x|\hat{g}_{u,\theta}^j(x)|\ll|u-i\theta|^\beta$  for $j=1,2,3,4$. 

Again $I_1=\Gamma(1-\beta)$.
To estimate $I^j$, $j=1,2,3,4$, we proceed similarly to the case (ii) of the proof of  Lemma~\ref{lem-zestimate}. Put $y=u/\theta$, so $ y\leq 1$. 
Substituting  $\sigma=\theta x$,
 \begin{align*}
\theta I^j&=
\int_0^\infty e^{-\sigma y}  e^{i\sigma}(1-G(\sigma/\theta))\hat{g}_{u,\theta}^j(\sigma/\theta)\,d\sigma.\end{align*}
Since $\sigma\mapsto e^{-\sigma y} (1-G(\sigma/\theta))\hat{g}_{u,\theta}^j(\sigma/\theta)$   is
decreasing for each fixed value of $u$ and $\theta$, 
\begin{align*}
\int_0^\infty e^{-\sigma y} \cos\sigma (1-G(\sigma/\theta))\hat{g}_{u,\theta}^j(\sigma/\theta)\,d\sigma&\leq \int_0^{\pi/2} e^{-\sigma y} \cos\sigma(1-G(\sigma/\theta))\hat{g}_{u,\theta}^j(\sigma/\theta)\,d\sigma\\&
\ll \theta^{2\beta}\int_0^{\pi/2} \sigma^{-\beta}\, d\sigma\ll \theta^{2\beta}.
\end{align*}
 The integral with $\cos$ replaced by $\sin$ can be treated similarly, so 
$|(u-i\theta)I^j|\ll \theta|I_j|\ll \theta^{2\beta}\le |u-i\theta|^{2\beta}$.

It remains to estimate $I_2$.
If $q<1$, define $H_2(x)=[x]^{-\beta}-x^{-\beta}+c([x])$.
Then
\[
I_2=-\int_0^\infty e^{-(u-i\theta)x}b([x])\,dx+
\int_0^\infty e^{-(u-i\theta)x}H_2(x)\,dx=-I_2'+I_2''.
\]
Clearly, $I_2''=O(1)$ since $H_2$ is integrable.
Let $y=u/\theta$, so $ y\leq 1$.  Substituting  $\sigma=\theta x$,
\begin{align*}
\theta I_2'=\int_0^\infty e^{-\sigma y}  e^{i\sigma} b([\sigma/\theta])\,d\sigma.
\end{align*}
Suppose for definiteness that $b$ is positive and decreasing.
Since $\sigma\mapsto   e^{-\sigma y} b([\sigma/\theta])$   is
decreasing for each fixed value of $u$ and $\theta$, we have
\[
\int_0^\infty e^{-\sigma y}  \cos\sigma\, b([\sigma/\theta])\, d\sigma\leq \int_0^{\pi/2} e^{-\sigma y} \cos\sigma\, b([\sigma/\theta])\, d\sigma\ll \theta^q  \int_0^{\pi/2} \sigma^q\, d\sigma\ll\theta^q.
\]
The integral with $\cos$ replaced by $\sin$ is treated similarly.
Hence $|(u-i\theta)I_2|\ll \theta|I_2|\ll|u-i\theta|^q$.
The proof for $q=1$ is identical, except that in the last step
\[
\int_0^\infty e^{-\sigma y}  \cos\sigma\, b([\sigma/\theta])\, d\sigma\
\le \theta+\theta\int_\theta^{\pi/2}\sigma^{-1}\,d\sigma\ll \theta\log\frac{1}{\theta} \ll |u-i\theta|\log\frac{1}{|u-i\theta|}.
\]
If $q>1$, then we proceed as in case~(i).
\end{proof}

\begin{pfof}{Lemma~\ref{lem-zest-error}} 
Taking $q=2\beta$, it follows 
from Lemma~\ref{lem-lambda} that $(1-\lambda(z))^{-1}P$ has the desired expansion.   By~\eqref{exprT}, Corollary~\ref{cor-H1} and Proposition~\ref{prop-PQ},
$\|T(z)-(1-\lambda(z))^{-1}P\|=O(1)$.
\end{pfof}

\begin{pfof}{Lemma~\ref{lemma-HOT-beta0}} 
The basic argument is similar to the one used in the proof of Lemma~\ref{lem-zest-error}, simplified by the fact that the various integrals are absolutely convergent.   We give the details for the more difficult case (b).
For notational convenience, we write $\lambda(u)$ instead
of $\lambda(e^{-u})$.

Let $G(x)=\mu(\varphi\le x)$ so that $1-G(x)=\ell(x)^{-1}+H(x)$, where $\ell\in O\Pi_{\hat{\ell}}$ and $H(x)=O(\ell(x)^{-2}\hat\ell(x))$. Write
$\lambda(u)=1+\int_0^\infty (e^{-ux}-1)\hat{v}_u(x)\,dG(x)$, where $|\hat{v}_u-1|_\infty\ll \ell(1/u)^{-1}$ as $u\to 0$.  Put $d\hat{G}_u=\hat{v}_u dG$. 
Integrating by parts,
\begin{align*}
1-\lambda(u) & =u\int_0^\infty e^{-ux}g_u(x)(1-G(x))\,dx
\\ & =u\int_0^\infty e^{-ux}g_u(x)\frac{1}{\ell(x)}\,dx +u\int_0^\infty e^{-ux}g_u(x)H(x)\,dx,
\end{align*}
where  $|g_u(x)-1|_\infty\ll \ell(1/u)^{-1}$ as $u\to 0$. Thus,
\begin{align*}
1-\lambda(u)&=\frac{1}{\ell(1/u)}\int_0^\infty e^{-\sigma}\frac{\ell(1/u)}{\ell(\sigma/u)}\,d\sigma +\int_0^\infty e^{-\sigma}(g_u(\sigma/u)-1)\frac{1}{\ell(\sigma/u)}\,d\sigma \\&+\int_0^\infty e^{-\sigma}g_u(\sigma/u)H(\sigma/u)\,d\sigma
=I_1+I_2+I_3.
\end{align*}
Now,
\begin{align*}
I_1&=\frac{1}{\ell(1/u)}\int_0^\infty e^{-\sigma}\,d\sigma +\frac{1}{\ell(1/u)}\int_0^\infty e^{-\sigma}\Big(\frac{\ell(1/u)}{\ell(\sigma/u)}-1\Big)\,d\sigma\\&
=\frac{1}{\ell(1/u)}+\frac{\hat{\ell}(1/u)}{\ell(1/u)^2}\int_0^\infty e^{-\sigma}\frac{\ell(1/u)}{\ell(\sigma/u)}\Big(\frac{\ell(1/u)-\ell(\sigma/u)}
{\hat{\ell}(1/u)}\Big)\,d\sigma.
\end{align*}
By Potter's bounds, for any fixed $\delta>0$, $\ell(1/u)\ell(\sigma/u)^{-1}\ll \sigma^\delta+\sigma^{-\delta}$. Also, by~\cite[Theorem 3.8.6]{BGT} (which
is the analogue of Potter's bounds for de Haan functions), for any fixed $\delta>0$, $\hat{\ell}^{-1}(1/u)|\ell(1/u)-\ell(\sigma/u)|\ll\sigma^\delta+\sigma^{-\delta}$. Hence,
\[
I_1=\frac{1}{\ell(1/u)}+O\Big(\frac{\hat{\ell}(1/u)}{\ell(1/u)^2}\Big).
\]
Next,
\[
|I_2|\ll \frac{1}{\ell(1/u)}\int_0^\infty e^{-\sigma}\frac{1}{\ell (\sigma/u)}\,d\sigma\ll \frac{1}{\ell(1/u)^2}\int_0^\infty e^{-\sigma}\frac{\ell(1/u)}{\ell(\sigma/u)}
d\sigma\ll \frac{1}{\ell(1/u)^2},\] 
and similarly
$|I_3|=O(\ell(1/u)^{-2}\hat\ell(1/u))$.
Putting all these together, we have shown that
$1-\lambda(u)=\ell(1/u)^{-1}+O(\ell(1/u)^{-2}\hat\ell(1/u))$.
Hence
\[
(1-\lambda(u))^{-1}=\ell(1/u)+O(\hat{\ell}(1/u)). 
\]
Again, the result follows by~\eqref{exprT}, Corollary~\ref{cor-H1} and
Proposition~\ref{prop-PQ}.~\end{pfof}

\section{Some contour integrals}\label{app-cont-int}

\begin{prop}\label{first_cont}Let $\beta\in (0,1)$. Then for every $u>0$ and $\theta\neq 0$ fixed, we have
\begin{align*}
 \int_0^R \frac{e^{-(u-i\theta)x}}{[(u-i\theta)x]^{\beta}}(u-i\theta)\,dx=\Gamma(1-\beta)+O(R^{-\beta}), \mbox{ as } R\to\infty.
\end{align*}
 
\end{prop}

\begin{proof} Write

\[ \int_0^R e^{-(u-i\theta)x}[(u-i\theta)x]^{-\beta}(u-i\theta)\,dx=\int_\Gamma e^{-w}w^{-\beta}\,dw,\] where $\Gamma$ is a line segment of length $R$ emanating from $0$ in the fourth quadrant.
 The angle formed by $\Gamma$  with the positive real axis is $\phi:=-\arg(u-i\theta)
\in[0,\pi/2]$. 

Define the arcs $S_\delta=\{\delta e^{i\psi}:-\phi\le\psi\le0\}$,
$S_R=\{R e^{i\psi}:-\phi\le\psi\le0\}$, and let
$L$ be the line segment from $\delta$ to $R$ along the real axis. By Cauchy's theorem,
\begin{equation*}
\int_\Gamma e^{-w}w^{-\beta}\,dw=\lim_{\delta\to 0}\Big(\int_{S_\delta} +\int_L-\int_{S_R} \Big) e^{-w}w^{-\beta}\,dw.
\end{equation*}
On $S_{\delta}$,
 \begin{align*}
\Bigl|\int_{S_\delta}e^{-w}w^{-\beta} \,dw\Bigr|=\Bigl|\int_{-\phi}^0  e^{-\delta e^{i\psi}} (\delta e^{i\psi})^{-\beta}  \delta i e^{i\psi}\,d\psi\Bigr|\leq\int_{-\pi/2}^0 \delta^{1-\beta}\,d\psi\leq (\pi/2) \delta^{1-\beta}.
\end{align*}
On $S_R$,
\begin{align*}\Bigl|\int_{S_R} e^{-w}w^{-\beta} \,dw\Bigr|=&
\Bigl|\int_{-\phi}^0  e^{-Re^{-\psi}} (R e^{i\psi})^{-\beta} R i e^{i\psi}\,d\psi\Bigr|\le
R^{1-\beta}\int_0^{\pi/2}  e^{-R \cos\psi}\,d\psi\end{align*} 
 Since $\cos\psi>1-2\psi/\pi$ for $\psi\in(0,\pi/2)$ (just draw the graph of $\cos\psi$ and $1-2\psi/\pi$), 
\begin{align*}\Bigl|\int_{S_R} e^{-w}w^{-\beta} \,dw\Bigr|&\le  R^{1-\beta}\int_0^{\pi/2}e^{-R(1-2\psi/\pi)}\,d\psi 
= \frac{\pi}{2 R^{\beta}}(1-e^{-R}).
\end{align*}
Also,
$\int_L e^{-w}w^{-\beta}\,dw =\int_{\delta}^R e^{-t}t^{-\beta}\,dt$. Thus, letting $\delta\to0$,
$\int _\Gamma  e^{-w}w^{-\beta} \,dw=\int_0^R e^{-t}t^{-\beta}\,dt +O(R^{-\beta})=\int_0^\infty e^{-t}t^{-\beta}\,dt +O(R^{-\beta})$, which ends the proof.~\end{proof}

\begin{prop}\label{second_cont}Let $\beta\in (0,1)$. Then
 \[\int_{-\infty}^\infty\frac{1}{(1-i\sigma)^{\beta+1}} e^{-i\sigma}\,d\sigma=\frac{2\pi}{e}\frac{1}{\Gamma(1+\beta)}.\]
\end{prop}

\begin{proof} First, write
 \[\int_{-\infty}^\infty\frac{1}{(1-i\sigma)^{\beta+1}} e^{-i\sigma}\,d\sigma=\frac{i}{e}\int_{\gamma}\frac{e^w}{w^{\beta+1}}\,dw,\] where the contour $\gamma=\{\Re w=1\}$ is traversed downwards.

We estimate $\int_\gamma e^w w^{-\beta}\,dw$   and then integrate by parts to complete the proof.    The diagram below shows a simple closed 
contour consisting of 
$\gamma$ together with oriented curves $C$, $L^\pm$, $M^\pm$, $S^\pm$, 

\begin{figure}[htb]
\centerline{
\mbox{\epsfig{file=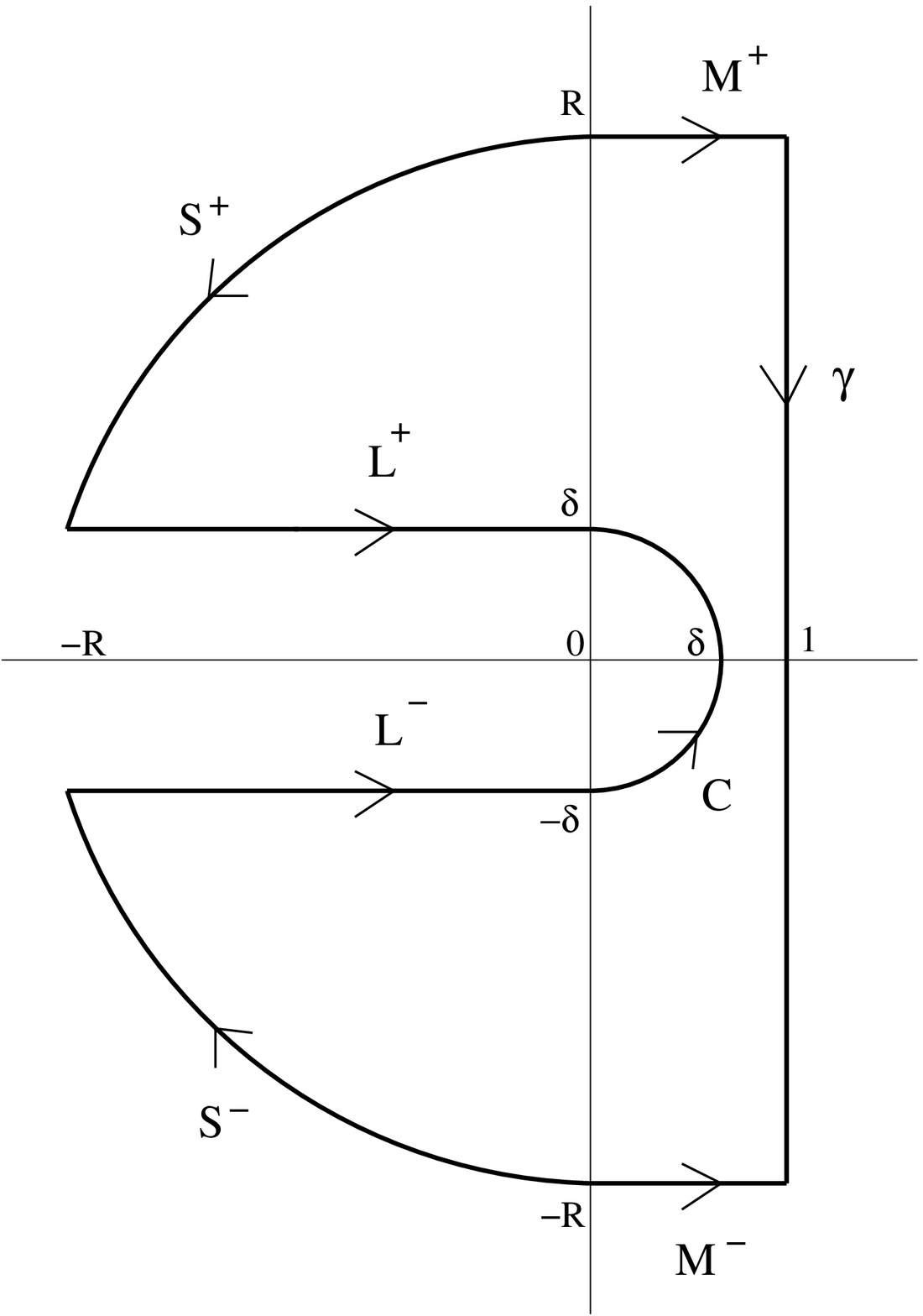,width=6cm}}
}
\end{figure}

By Cauchy's theorem, 
\begin{equation*}\int_{\gamma}\frac{e^w }{w^{\beta}}\,dw=\lim_{R\to\infty}\lim_{\delta\to0^+}\Bigl(-\int_{M^+}+ \int_{S^+}+ \int_{L^+}-\int_C-\int_{L^-}-\int_{S^-}
+\int_{M^-}\Bigr) \frac{e^w }{w^{\beta}}\,dw.
\end{equation*}
Now $C=\{w= \delta e^{i\theta}: \theta\in (-\pi/2,\pi/2)\}$, and so
\begin{align*}\Bigl|\int_C e^w w^{-\beta}\,dw\Bigr| =\Bigl|\int_{-\pi/2}^{\pi/2} \exp(\delta e^{i\theta})\delta^{1-\beta}e^{i\theta(1-\beta)}\,d\theta\Bigr|\ll\delta^{1-\beta}.
\end{align*} 
On  $M^\pm$, we have 
\[
\Bigl|\int_{M^\pm}e^w w^{-\beta}\,dw\Bigr|=\Bigl|\int_0^1e^te^{iR}(t\pm iR)^{-\beta}\,dt\Bigr|\leq R^{-\beta}\int_0^1 e^t \,dt\ll R^{-\beta}. 
\]
On  $S^\pm$, we write
\begin{align*}
 \Bigl|\int_{S^\pm}e^w w^{-\beta}\,dw\Bigr|&=\Bigl|\int_{\pi/2}^{\pi-\epsilon}e^{Re^{\pm i\theta}}(R e^{\pm i\theta})^{-\beta} R i e^{\pm i\theta}\,d\theta\Bigr|\leq
R^{1-\beta}\int_{\pi/2}^{\pi}e^{R \cos\theta}\,d\theta\\&= 
R^{1-\beta}\int_0^{\pi/2}e^{-R \sin\theta}\,d\theta.
\end{align*}
(Here, $\epsilon=\epsilon(R,\delta)$ is chosen so that $S^\pm$ connects with 
$L^\pm$ as shown.) But 
 $\sin\theta>\theta/\pi$
for $\theta\in (0,\pi/2)$. Hence, 
$|\int_{S^\pm}e^w w^{-\beta}\,dw|\le R^{1-\beta}\int_0^{\pi/2}e^{-R\theta/\pi}\,d\theta=\pi R^{-\beta}(1-e^{-R/2})$.

We have shown that 
$\int_{\gamma}\frac{e^w }{w^{\beta}}\,dw=\lim_{R\to\infty}\lim_{\delta\to0^+}\Bigl( \int_{L^+}\frac{e^w }{w^{\beta}}\,dw- \int_{L^-}\frac{e^w }{w^{\beta}}\,dw\Bigr)$.  Now
\begin{align*}\int_{L^\pm}e^w w^{-\beta}\,dw=\int_{-R}^0 e^t e^{\pm i\delta}(t^2+\delta^2)^{-\beta/2} e^{-i\beta\arg(t\pm  i\delta)}\,dt.
\end{align*}
Since $t<0$,  $\lim_{\delta\to0^+}\arg(t\pm i\delta)\to\pm\pi$. 
Hence, the integrand converges pointwise to $e^{\mp i\pi\beta}e^t |t|^{-\beta}$ as $\delta\to0^+$.
Moreover, the integrand is bounded by $e^t |t|^{-\beta}$ 
which is integrable on $[-R,0]$. It follows from the  DCT that
$\lim_{\delta\to0^+}\int_{L^\pm}e^w w^{-\beta}\,dw= e^{\mp i\pi\beta}\int_{-R}^0 e^t|t|^{-\beta}\,dt=  e^{\mp i\pi\beta}\int_0^R e^{-t}t^{-\beta}\,dt$.
Hence,
\begin{align*}\int_{\gamma}\frac{e^w }{w^{\beta}}\,dw= -2i\sin\pi\beta\lim_{R\to\infty}\int_0^R e^{-t}t^{-\beta}\,dt =-2\pi i\frac{1}{\Gamma(\beta)},
\end{align*}
where we have used  the formula $\frac{\sin\pi\beta}{\pi}=\frac{1}{\Gamma(1-\beta)\Gamma(\beta)}$.

Finally,
\begin{align*} 
\frac{i}{e}\int_\gamma\frac{e^w }{w^{\beta+1}}\,dw
&=-\frac{i}{e\beta}\frac{e^w }{w^\beta}\Big|_{w=1-i\sigma,\sigma\to -\infty}^{w=1-i\sigma,\sigma\to \infty}+ \frac{i}{e\beta}\int_\gamma \frac{e^w }{w^\beta} = \frac{2\pi}{e\beta}\frac{1}{\Gamma(\beta)},
\end{align*}
as required.
\end{proof}

\begin{cor}\label{contb}For any $\rho>0$ $\gamma\in(0,1)$,
\[\int_{-1/n^\gamma}^{1/n^\gamma}\frac{e^{-in\theta}}{(\frac1n-i\theta)^{\rho+1}}\,d\theta=\frac{2\pi}{e}\frac{n^{\rho}}{\Gamma(1+\rho)}+O(n^{\rho\gamma}).\]
\end{cor}

\begin{proof} 
Compute that
\begin{align*}\int_{-1/n^\gamma}^{1/n^\gamma}\frac{e^{-in\theta}}{(\frac1n-i\theta)^{\rho+1}}\,d\theta&=
n^\rho\int_{-n^{1-\gamma}}^{n^{1-\gamma}}\frac{e^{-i\sigma}}{(1-i\sigma)^{\rho+1}}\, d\sigma \\ &
=n^\rho\int_{-\infty}^\infty\frac{e^{-i\sigma}}{(1-i\sigma)^{\rho+1}}\, d\sigma
- n^\rho\Big(\int_{-\infty}^{-n^{1-\gamma}}+\int_{n^{1-\gamma}}^{\infty}\Big)\frac{ e^{-i\sigma}}{(1-i\sigma)^{\rho+1}}\,d\sigma \\
& =\frac{2\pi}{e}\frac{n^\rho}{\Gamma(1+\rho)}+O\Bigl(n^\rho \int_{n^{1-\gamma}}^{\infty}\frac{1}{\sigma^{\rho+1}}\,d\sigma\Bigr)
\\
& =\frac{2\pi}{e}\frac{n^\rho}{\Gamma(1+\rho)}+O(n^{\rho\gamma}),\end{align*}
where we have used Proposition~\ref{second_cont}.
\end{proof}

\section{Tail sequences for (\ref{eq-LSV}) and (\ref{eq-LSV0})}
\label{app-tails}

The following proposition is an improved version of~\cite[Proposition 11.9]{MT}.
Recall that $h$ denotes the density for the measure $\mu$.

\begin{prop}\label{prop-tail-PM} Suppose that $f:[0,1]\to [0,1]$ is given as in~\eqref{eq-LSV} with
$\beta=1/\alpha\in(0,1)$.  
Let $C$ be a compact subset of $(0,1]$.  Then there exists $Y\subset(0,1]$ compact
with $C\subset Y$, such that the first return function $\varphi:Y\to\Z^+$ satisfies
$\mu(\varphi>n)= cn^{-\beta}-b(n)+O(n^{-(\beta+1)})$, where 
$c=\frac14\beta^\beta h(\frac12)$ and $b(n)$ is a decreasing function satisfying $b(n)=O(n^{-2\beta})$.
\end{prop}

\begin{proof}
First, let $Y=[\frac12,1]$.
Let $x_n\in(0,\frac12]$ be the sequence with $x_1=\frac12$ and $x_n=fx_{n+1}$
so $x_n\to0$.   It is well known (see for instance~\cite{LiveraniSaussolVaienti99}) that $x_n\sim \frac12\beta^\beta n^{-\beta}$. Since $x_{n}-x_{n+1}=
fx_{n+1}-x_{n+1}=O(n^{-(\beta+1)})$, we have
$x_n=\frac12\beta^\beta n^{-\beta}+O(n^{-(\beta+1)})$.

The density $h$ is globally Lipschitz on $(\epsilon,1]$
for any $\epsilon>0$ (see for example~\cite{Hu04}
or~\cite[Lemma~2.1]{LiveraniSaussolVaienti99}).  Furthermore, 
it follows from~\cite[Lemma~2]{Thaler02}, see also~\cite{Murray10}, that $h$
is decreasing.   Hence for $x\in[\frac12,1]$ we can write $h(x)=h(\frac12)-\tilde h(x)$ where $\tilde h$ is positive and Lipschitz.

Set $y_n=\frac12(x_n+1)$ (so $fy_n=x_n$).
Then $\varphi=n$ on $[y_n,y_{n-1}]$, hence $\{\varphi>n\}=[\frac12,y_n]$.
It follows that
\[
\SMALL
\mu(\varphi>n)=\int_{1/2}^{y_n} h(x)\,dx=\frac12 x_nh(\frac12)-b(n)=
\frac14\beta^\beta h(\frac12)n^{-\beta}+O(n^{-(1+\beta)})-b(n), 
\]
where $b(n)=\int_{1/2}^{y_n}\tilde{h}(x)\,dx$ is decreasing. Moreover, $b(n)\ll(y_n-\frac12)^2=\frac14 x_n^2
\ll n^{-2\beta}$.

The same estimates are obtained by inducing on $Y=[x_q,1]$ for any fixed $q\ge0$.
\end{proof}

\begin{prop}\label{prop-tail-PM0} Suppose that $f:[0,1]\to [0,1]$ is given as in~\eqref{eq-LSV0}.  
Let $C$ be a compact subset of $(0,1]$.  Then there exists $Y\subset(0,1]$ compact
with $C\subset Y$, such that the first return function $\varphi:Y\to\Z^+$ satisfies
$\mu(\varphi>n)= c\log^{-1}n+O(\log^{-2} n)$,
where $c=\frac 12 h(\frac12)$.
\end{prop}

\begin{proof}
First, let $Y=[\frac12,1]$.
Let $x_n\in(0,\frac12]$ be the sequence with $x_1=\frac12$ and $x_n=fx_{n+1}$
so $x_n\to0$. We claim that $x_n\sim\log^{-1}n$, in accordance with~\cite[Remark~2, p. 94]{Thaler83}.  By (\ref{eq-LSV0}), 
\begin{align*} 
e^{1/x_j}&=\exp\Bigl(\frac{1}{x_{j+1}}(1+x_{j+1} e^{-1/x_{j+1}})^{-1}\Bigr)
\\ & =\exp\Bigl(\frac{1}{x_{j+1}}(1-x_{j+1} e^{-1/x_{j+1}}+O(x_{j+1}^2 e^{-2/x_{j+1}}))\Bigr)\\&=
 e^{1/x_{j+1}}\exp(-e^{-1/x_{j+1}}+O(x_{j+1} e^{-2/x_{j+1}})) \\ &  = 
e^{1/x_{j+1}}(1- e^{-1/x_{j+1}}+O(e^{-2/x_{j+1}}))  =
e^{1/x_{j+1}}-1+O(e^{-1/x_{j+1}}).
\end{align*}
Hence $e^{1/x_{j+1}}-e^{1/x_j}=1+O(e^{-1/x_{j+1}})$.  
Summing from $j=1$ to $n-1$, 
\[
e^{1/x_n}=e^{1/x_1}+n-1+O\Bigl(\sum_{j=2}^{n}e^{-1/x_{j}}\Bigr).
\]
Since $x_n\to0$, for $n$ sufficiently large we have 
$e^{1/x_n}\in(\frac12 n,\frac32 n)$ .
Hence $1/x_n\in(\log n+\log\frac12,\log n+\log\frac12)$ and 
$x_n\sim\log^{-1}n$ verifying the claim. 
Moreover, $x_{n}-x_{n+1}=
fx_{n+1}-x_{n+1}=O(1/(n\log^2 n))$, so $x_n=\log^{-1}n+O((n\log^2n)^{-1})$.

Again, the density $h$ is globally Lipschitz on $(\epsilon,1]$
for any $\epsilon>0$ (see~\cite{Thaler83}).  Thus,
in the notation of the proof of Proposition~\ref{prop-tail-PM},
\[
\SMALL
\mu(\varphi>n)=\int_\frac12^{y_n} h(x)\,dx = (y_n-\frac12)h(\frac12)+O(y_n-\frac12)^2 = \frac12 h(\frac12)\log^{-1}n+O(\log^{-2}n),
\]
as required.

The same estimates are obtained by inducing on $Y=[x_q,1]$ for any fixed $q\ge0$.
\end{proof}

\paragraph{Acknowledgements}
The research of IM and DT was supported in part by EPSRC Grant EP/F031807/1.
We wish to thank Roland Zweim\"uller
for useful conversations.


\begin{thebibliography}{10}

\bibitem{Aaronson86}
J.~Aaronson. Random {$f$}-expansions. \emph{Ann. Probab.} \textbf{14} (1986)
  1037--1057.

\bibitem{Aaronson}
J.~Aaronson. \emph{{An Introduction to Infinite Ergodic Theory}}. Math. Surveys
  and Monographs \textbf{50}, Amer. Math. Soc., 1997.

\bibitem{AaronsonDenker90}
J.~Aaronson and M.~Denker. {Upper bounds for ergodic sums of infinite measure preserving transformations}. \emph{Trans. Amer. Math. Soc.} \textbf{319}
  (1990) 101--138.

\bibitem{AaronsonDenker01}
J.~Aaronson and M.~Denker. {Local limit theorems for partial sums of stationary
  sequences generated by Gibbs-Markov maps}. \emph{Stoch. Dyn.} \textbf{1}
  (2001) 193--237.


\bibitem{ADF92}
J.~ Aaronson, M.~ Denker and A.~M.~ Fisher. Second order ergodic theorems for ergodic transformations
of infinite measure spaces.\emph{ Proc. Amer. Math. Soc.} \textbf{114} (1992), 115--127.

\bibitem{ADU93}
J.~Aaronson, M.~Denker and M.~Urba{\'n}ski. Ergodic theory for {M}arkov fibred
  systems and parabolic rational maps. \emph{Trans. Amer. Math. Soc.}
  \textbf{337} (1993) 495--548.  

\bibitem{Aljancic}
S. ~Aljan\v{c}i\'c, R. ~Bojani\'c and M. ~Tomi\'c. Slowly varying functions with remainder and their applications in analysis. \emph{Serbian Acad. Sci. Acts Monographs}
  \textbf{467} (1974), Beograd.

\bibitem{Bingham71}
N.~H. Bingham. Limit theorems for occupation times of {M}arkoff processes.
  \emph{Z. Wahrscheinlichkeistheorie verw. Geb.} \textbf{17} (1971) 1--22.

\bibitem{BGT}
N.~H. Bingham, C.~M. Goldie and J.~L. Teugels. \emph{Regular variation}.
  Encyclopedia of Mathematics and its Applications \textbf{27}, Cambridge
  University Press, Cambridge, 1987.


\bibitem{DarlingKac57}
D.~A. Darling and M.~Kac. On occupation times for {M}arkoff processes.
  \emph{Trans. Amer. Math. Soc.} \textbf{84} (1957) 444--458.

\bibitem{Erickson}
K.~B. Erickson. Strong renewal theorems with infinite mean.
\emph{Trans. Amer. Math. Soc.}
\textbf{151} (1970) 263--291.


\bibitem{Feller66}
W.~Feller. \emph{{An Introduction to Probability Theory and its Applications,
  II}}. Wiley, New York, 1966.

\bibitem{Freud51}
G. Freud.  Restglied eines Tauberschen Satzes. I. 
\emph{Acta Math. Acad. Sci. Hungar.} \textbf{2} (1951) 299--308. 

\bibitem{GarsiaLamperti62}
A.~Garsia and J.~Lamperti. A discrete renewal theorem with infinite mean.
 \emph{Comment. Math. Helv.} \textbf{37} (1962/1963) 221--234.

\bibitem{Gouezel04}
S.~Gou{\"e}zel. {Sharp polynomial estimates for the decay of correlations}.
  \emph{Israel J. Math.} \textbf{139} (2004) 29--65.

\bibitem{Gouezel05}
S.~Gou{\"e}zel. Berry-{E}sseen theorem and local limit theorem for non
  uniformly expanding maps. \emph{Ann. Inst. H. Poincar\'e Probab. Statist.}
  \textbf{41} (2005) 997--1024.


\bibitem{Geluk}
J.~Geluk. {$\pi$ regular variation.} \emph{Proc. Amer. Math. Soc.}
\textbf{82} (1981) 565--570.

\bibitem{GelukdeHaan}
J.~Geluk and L.~de Haan. On functions with small differences. \emph{Indag. Math.} \textbf{84} (1981) 187--194.


\bibitem{deHaan70}
L.~de Haan. {On Regular Variation and its Application to the Weak Convergence of
Sample Extremes.} \emph{Centre Tract} \textbf{32} (1970), Amsterdam.

\bibitem{deHaan}
L.~de Haan. {An Abel Tauber theorem for Laplace transform.} \emph{J. London Math. Soc} \textbf{13} (1976) 537--542.

\bibitem{Hu04}
H.~Hu. Decay of correlations for piecewise smooth maps with indifferent fixed
  points. \emph{Ergodic Theory Dynam. Systems} \textbf{24} (2004) 495--524.

\bibitem{Ingham65}A.~Ingham. On Tauberian theorems. \emph{Proc. London Math. Soc.} \textbf{14A} (1965) 157--173.


\bibitem{Jordan}G.~Jordan.  Regularly varying functions and convolutions with real kernels.  \emph{Trans. Amer. Math.  Soc.} \textbf{194} (1974) 177--194.

\bibitem{Karamata30}
J. Karamata.
\"{U}ber die Hardy-Littlewoodschen Umkehrungen des Abelschen Stetigkeitssatzes. 
\emph{Math. Z.} \textbf{32} (1930) 319--320. 

\bibitem{Karamata31}
J. Karamata.  Neuer Beweis und Verallgemeinerung der Tauberschen S\"{a}tze, welche die Laplacesche und Stieltjessche Transformation betreffen.
\emph{J. Reine Angwe. Math.} \textbf{164} 27--39.

\bibitem{Korevaar54a} J.~Korevaar. {A very general form of Littlewood's theorem}. 
\emph{Indag. Math} \textbf{16} (1954) 36--45.

\bibitem{Korevaar54b} J.~Korevaar. {Another numerical Tauberian theorem for power series}.
\emph{Indag. Math.} \textbf{16} (1954) 45--56.

\bibitem{Korevaar} J.~Korevaar.
\emph{Tauberian theory. A century of developments}.
 Springer Verlag Berlin Heidelberg, (2004).

\bibitem{Lamperti62}
J.~Lamperti. An invariance principle in renewal theory. \emph{Ann. Math.
  Statistics.} \textbf{33} (1962) 685--696.


\bibitem{LiveraniSaussolVaienti99}
C.~Liverani, B.~Saussol and S.~Vaienti. {A probabilistic approach to
  intermittency}. \emph{Ergodic Theory Dynam. Systems} \textbf{19} (1999)
  671--685.

\bibitem{MT} I.\ Melbourne and D.\ Terhesiu. \emph{Operator renewal theory and mixing rates for dynamical
systems with infinite measure}.  Preprint 2010.

\bibitem{Murray10}
R.~Murray. Ulam's method for some non-uniformly expanding maps. \emph{Discrete Cont. Dyn. Syst.} \textbf{26} (2010) 1007--1018.

\bibitem{Omey85}
E.~Omey. {Tauberian theorems with remainder}. \emph{J. London Math. Soc.}
\textbf{32} (1985) 116--132.


\bibitem{PomeauManneville80}
Y.~Pomeau and P.~Manneville. Intermittent transition to turbulence in
  dissipative dynamical systems. \emph{Comm. Math. Phys.} \textbf{74} (1980)
  189--197.


\bibitem{Sarig02}
O.~M. Sarig. {Subexponential decay of correlations}. \emph{Invent. Math.}
\textbf{150} (2002) 629--653.

\bibitem{Subhankulov60}
M. A. Subhankulov.
Tauberian theorems with remainder term. 
\emph{Mat. Sbornik} \textbf{52(94)} (1960) 823--846.  {\em Amer. Math. Soc. Transl.} \textbf{26} (1963) 311--338. 

\bibitem{Thaler83}
M.~Thaler.  {Transformations on [0,1] with infinite invariant measures}. \emph{Israel. J. Math.}
\textbf{46} (1983) 67--96 

\bibitem{Thaler95}
M.~Thaler. A limit theorem for the {P}erron-{F}robenius operator of
  transformations on {$[0,1]$} with indifferent fixed points. \emph{Israel J.
  Math.} \textbf{91} (1995) 111--127.

\bibitem{Thaler02}
M.~Thaler. A limit theorem for sojourns near indifferent fixed points of one-dimensional maps. \emph{Ergod. Th. Dynam. Syst.} \textbf{22} (2002)
 1289--1312.

\bibitem{ThalerZweimuller06}
M.~Thaler and R.~Zweim{\"u}ller. Distributional limit theorems in infinite
  ergodic theory. \emph{Probab. Theory Related Fields} \textbf{135} (2006)
  15--52.  



\bibitem{Zweimuller98}
R.~Zweim{\"u}ller. Ergodic structure and invariant densities of non-{M}arkovian
  interval maps with indifferent fixed points. \emph{Nonlinearity} \textbf{11}
  (1998) 1263--1276.

\bibitem{Zweimuller00}
R.~Zweim{\"u}ller. Ergodic properties of infinite measure-preserving interval
  maps with indifferent fixed points. \emph{Ergodic Theory Dynam. Systems}
  \textbf{20} (2000) 1519--1549.

\end{thebibliography}
\end{document}